\documentclass[a4paper, 11pt]{article}
\usepackage{amssymb}
\usepackage{amsmath}
\usepackage{enumitem}
\usepackage{amsthm}
\usepackage{amsfonts,latexsym,amstext}
\usepackage[T1]{fontenc}
\usepackage [latin1]{inputenc}
\usepackage{color}
\usepackage{bbm}

\newtheorem{theorem}{Theorem}[section]
\newtheorem*{theorem*}{Theorem}
\newtheorem{definition}[theorem]{Definition}

\newtheorem{lemma}[theorem]{Lemma}
\newtheorem{proposition}[theorem]{Proposition}
\newtheorem{corollary}[theorem]{Corollary}

\newcommand{\R}{{\mathbb R}}

\newcommand{\N}{{\mathbb N}}

\newcommand{\F}{{\mathcal F}}

\newcommand{\E}{{\mathbb E}}

\newcommand{\B}{{\mathcal B}}
\newcommand{\cG}{{\mathcal G}}
\newcommand{\cN}{{\mathcal N}}
\newcommand{\cL}{{\mathcal L}}
\newcommand{\cD}{{\mathcal D}}
\newcommand{\rmP}{{\rm P}}

\title
{A Strong Law of Large Numbers for Super-critical Branching Brownian Motion with Absorption}
\author{Oren Louidor\thanks{oren.louidor@gmail.com, saglietti.s@technion.ac.il.} \\ Technion, Israel \and Santiago Saglietti\footnotemark[1] \\Technion, Israel}
\date{} 
\begin{document}
\maketitle

\begin{abstract} We consider a (one-dimensional) branching Brownian motion process with a general offspring distribution having at least two moments, and in which all particles have a drift towards the origin where they are immediately absorbed. It is well-known that the population survives forever with positive probability if and only if the branching rate is sufficiently large. We show that throughout this super-critical regime, the number of particles inside any fixed set normalized by the mean population size converges to an explicit limit, almost surely and in $L^1$. As a consequence, we get that almost surely on the event of eternal survival, the empirical distribution of particles converges weakly to the (minimal) quasi-stationary distribution associated with the Markovian motion driving the particles. This proves a result of Kesten in~\cite{kesten1978} from 1978, for which no proof was available until now. 
\end{abstract} 

\section{Introduction and Results}
Given some fixed $c > 0$, let $X=(X_t)_{t \geq 0}$ be a Brownian motion with drift $-c$ and variance coefficient $1$, which is absorbed upon reaching the origin, i.e. $X$ is the process given by 
\begin{equation} \label{eq:defx}
X_t := X_0 - c (t \wedge H_0) + W_{t \wedge H_0}
\end{equation} for each $t$, where $W$ is a standard Brownian motion on $\R$ and $H_0:=\inf \{ s \geq 0 : X_0 - cs + W_s = 0 \}.$  Now, consider the following branching dynamics associated with $X$:
	\begin{enumerate}[wide, labelindent=0pt]
		\item [i.] The dynamics starts with a single particle, located initially at some $x \geq 0$, whose position evolves randomly according to $X$. 
		\item [ii.] This initial particle branches at a fixed rate $r > 0$ (independently of the motion it describes) and, whenever it does so, it dies and gets replaced at its current position by an independent random number of particles $m$ having some fixed distribution $\mu$ on $\N_0$.
		\item [iii.] Starting from their birth position, now each of these $m$ new particles independently mimics the same stochastic behavior of its parent.
		\item [iv.] If a particle has $0$ children, then it simply dies and disappears from the dynamics.
	\end{enumerate}
We will call this the $(c,r,\mu)$-\textit{branching dynamics} associated with $X$ (or simply $(c,r,\mu)$-dynamics).

\medskip
Let us agree on the following notation to be used throughout the sequel:
\begin{itemize}[wide, labelindent=0pt]
	\item For each $t \geq 0$ we denote by $A_t$ the collection of all particles present in the dynamics at time $t$. 
	\item For any particle $u \in A_t$ and $0 \leq s \leq t$ we let $u_s$ be the position of the unique ancestor of $u$ (including $u$ itself) which belongs to $A_s$. Furthermore, 
	we will write $\overline{u}_t:=(u_s)_{s \in [0,t]}$ to denote its trajectory in the time interval $[0,t]$.
	\item We will write $\B_{(0,+\infty)}$ for the class of all Borel subsets of $(0,+\infty)$ and, for any given $t \geq 0$ and $B \in \B_{(0,+\infty)}$, use $N_t(B)$ to denote the sub-collection of particles in $A_t$ which lie inside $B$. 
	Also, to simplify the notation, in the sequel we will write $N_t$ instead of $N_t((0,+\infty))$. Observe that $A_t \setminus N_t$ contains exactly those particles which are alive at time $t$, but have already been absorbed.

	\item $|N_t(B)|$ shall indicate the cardinality of $N_t(B)$ (and, analogously, $|N_t|$ that of $N_t$).
	\item We will use the superscript $x$, e.g. $X^{(x)}_t$ or $N^{(x)}_t$, to indicate that the corresponding process starts at position $x$. Similarly, we shall use the subscript $x$, e.g. $P_x$ or $\E_x$, to indicate that the process involved in the corresponding probability or expectation starts at $x$.
	\item $\mu_1:=\E(m)$ and $\mu_2:= \E(m^2)$ will respectively denote the first and second moments of $\mu$.
\end{itemize}

Assuming $\mu_1 < \infty$ it is not difficult to show (see Lemma~\ref{lemma:asymp}) that for any $x \geq 0$ as $t \to \infty$,
\begin{equation} \label{eq:sc2} 
\E_x(|N_t|) \sim \frac{2 x e^{cx}}{\sqrt{\pi}c^2}\, t^{-\frac{3}{2}}e^{(r(\mu_1-1)-\frac{c^2}{2})t}  \,,
\end{equation} 
where $a(t) \sim b(t)$ means that $a(t)/b(t) \to 1$, which suggests that the positivity of the exponent coefficient $r(\mu_1-1)-c^2/2$ governs the possibility of survival for the $(c, r, \mu)$-dynamics. Indeed, by Markov's inequality, if $r(\mu_1-1)-c^2/2 \leq 0$ then \eqref{eq:sc2} above implies that the 
process must die out eventually with probability $1$. The other regime was first addressed by Kesten in his paper~\cite{kesten1978} from 1978:
	\begin{theorem*}[Kesten] If $\mu_1 < +\infty$ and $r(\mu_1-1)> \frac{c^2}{2}$ then the $(c,r,\mu)$-branching dynamics is super-critical, i.e. for all $x > 0$,
		\begin{equation} \label{eq:sc1}
		P_x( N_t > 0 \text{ for all }t \geq 0) > 0 \,.
		\end{equation}
	Moreover, if $\mu_2 < +\infty$ then there exists a random variable $D^{(x)}_\infty$ satisfying 
	\begin{equation}
	\label{eq:positivity}
	P_x( D_\infty > 0 \,|\, N_t > 0 \text{ for all }t \geq 0) = 1 \,,
	\end{equation}
	such that with probability $1$, simultaneously for all intervals $I \subseteq (0,+\infty)$ (including semi-infinite ones),
		\begin{equation} \label{eq:growth}
		\lim_{t \rightarrow +\infty} \frac{|N_t^{(x)}(I)|}{\E_x(|N_t(I)|)}=  D^{(x)}_\infty \,.
		\end{equation}
\end{theorem*}

Since the goal in~\cite{kesten1978} was to study the critical case: $r(\mu_1-1) = c^2/2$, the author provides no proof for the above, arguing that ``so far he had an ugly and complicated proof''. While the first part of the theorem, namely assertion~\eqref{eq:sc1}, is well-known by now (see, for example, Theorem 11 in~\cite{harris2006v1}), a proof for the remaining part has never been produced. The aim of the present work is, therefore, to provide the missing proof for the second part of the theorem.

\subsection{Main results}
We shall, in fact, prove a slightly stronger version of the theorem. To this end, let $(\rmP_t)_{t \geq 0}$ denote the semigroup associated with $X$, defined as
$$
\rmP_t(f)(x)=\E_x(f(X_t))
$$ for any nonnegative measurable function $f: \R_{\geq 0} \rightarrow \R_{\geq 0}$. 
It is well-known that (see, e.g., \cite{polak2012}), if we set $-\lambda:=-\frac{c^2}{2} < 0$, for each $t>0$ one has that $e^{-\lambda t}$ is an eigenvalue of $\rmP_t$ with corresponding right eigenfunction $h$ and left eigenmeasure $\nu$ given by (up to constant multiples):
\begin{equation} \label{eq:hnu}
h(x):= \frac{1}{\sqrt{2\pi \lambda^2}} xe^{cx} \quad \text{ and } \qquad d\nu(x):=2\lambda x e^{-cx}\mathbbm{1}_{(0,+\infty)}(x) dx \,.
\end{equation} That is, for every $t > 0$ and all nonnegative measurable $f:\R_{\geq 0}\rightarrow \R_{\geq 0}$,
$$
\rmP_t(h)=e^{-\lambda t}h \hspace{1cm}\text{ and }\hspace{1cm}\int \rmP_t(f)d\nu=e^{-\lambda t} \int fd\nu.
$$ 
Moreover, the eigenvectors $h$ and $\nu$ satisfy for all $x > 0$ and $B \in \B_{(0,+\infty)}$,
\begin{equation}
\label{eq:asymp}
h(x)=\lim_{t \rightarrow +\infty} t^{\frac{3}{2}}e^{\lambda t}P_x(X_t > 0) 
\quad \text{and} \quad
\lim_{t \rightarrow +\infty} P_x( X_t \in B | X_t > 0) = \nu(B)\,,
\end{equation} 
and are sometimes known as the {\em ground state} and {\em (minimal) quasi-stationary distribution} for $X$, respectively.

Now since $\rmP_t(h)=e^{-\lambda t}h$, the process $M^{(x)}:=(M_t^{(x)})_{t \geq 0}$, given for all $x > 0$ by
\begin{equation}
\label{eq:MDef}
M^{(x)}_t:=\frac{h(X_t^{(x)})}{h(x)}e^{\lambda t} \,,
\end{equation}
is a mean-one martingale with respect to $(\F^{(x)}_t)_{t \geq 0}$ - the natural filtration of $X^{(x)}$. By a standard application of the many-to-one lemma (see Lemma~\ref{lema:mt1} below), the same holds for the process $D^{(x)}:=(D_t^{(x)})_{t \geq 0}$, defined via
\begin{equation} \label{def:d}
D^{(x)}_t:= \frac{1}{h(x)}\sum_{u \in N_t^{(x)}} h(u_t)e^{-(r(\mu_1-1)-\lambda)t} = e^{-r(\mu_1-1)t} \sum_{u \in A^{(x)}_t} M^{(x)}_t(u),
\end{equation} where $M^{(x)}_s(u):=\frac{h(u_s)}{h(x)}e^{\lambda s}$ for $s \geq 0$. The process $D^{(x)}$ is called the \textit{additive martingale} associated with the $(c,r,\mu)$-branching dynamics. 

Being a nonnegative martingale, $D^{(x)}$ has an almost sure limit, which we shall denote by $D^{(x)}_\infty$. Our first result asserts that the convergence also holds in $L^1$ and that the limit $D^{(x)}_\infty$ is almost surely positive in the event of survival. Since in the sequel $D^{(x)}_\infty$ will play the same role as it did in Kesten's Theorem, this result corresponds to~\eqref{eq:positivity}.

\begin{proposition}\label{theo:main2} 
Assume that $r(\mu_1 - 1) > c^2/2$ and $\mu_2 < +\infty$. Then for all $x > 0$ we have,
	\begin{equation} \label{eq:d1}
	D_t^{(x)} \overset{L^1}{\underset{a.s.}{\longrightarrow}} D_\infty^{(x)}.
	\end{equation} Furthermore, $D^{(x)}_\infty$ is strictly positive almost surely in the event of survival, i.e. 
	\begin{equation}
	\label{eq:d2}
	P_x( D_\infty > 0 \,|\, |N_t| > 0 \text{ for all }t \geq 0) = 1.
	\end{equation}
\end{proposition}
We remark that equation \eqref{eq:d1} was already proved in \cite[Theorem 13]{Kyp2004} under weaker moment assumptions on $\mu$, but using a different method (Theorem 13 in~\cite{Kyp2004} treats the case $c \leq 0$, but its proof can be used essentially verbatim to handle the case $c > 0$ as well). Moreover, results similar to \eqref{eq:d2} can be found already in \cite{harris2006v1}. We can now state the principal result of this manuscript.
\begin{theorem}
\label{theo:main3}
Assume that $r(\mu_1 - 1) > c^2/2$, $\mu_2 < +\infty$ and let $x > 0$. Then with probability $1$ simultaneously for all $B \in \B_{(0,+\infty)}$ with $\nu(\partial B) = 0$ we have,
		\begin{equation} \label{eq:convmain2}
		\lim_{t \rightarrow +\infty} \frac{|N^{(x)}_t(B)|}{\E_x(|N_t|)} = \nu(B) \cdot D_\infty^{(x)} \,.
		\end{equation} 
	 The above convergence also holds in $L^1$ for any fixed $B \in \B_{(0,+\infty)}$ (not  necessarily with $\nu(\partial B)=0$).  
\end{theorem}

Proposition~\ref{theo:main2} and Theorem~\ref{theo:main3} admit three immediate corollaries. First, observe that \eqref{eq:d1} implies that $\E_x(D_\infty)=1$ and hence that $P_x(D_\infty > 0) > 0$. Since necessarily $D_\infty \equiv 0$ whenever $N^{(x)}_t$ dies out, we must have:
\begin{corollary}
For all $x > 0$ we have $P_x( N_t > 0 \text{ for all }t \geq 0)>0$.
\end{corollary}
\noindent This reproduces~\eqref{eq:sc1} in Kesten's Theorem.
Next, we use Theorem~\ref{theo:main3} and the fact that $\E_x(D_\infty)=1$
to conclude that $\E_x(|N_t(B)|) \sim \nu(B) \E_x(|N_t|)$ as $t \to \infty$. Plugging this back in~\eqref{eq:convmain2} then yields:
\begin{corollary}
Let $x > 0$. Then with probability $1$ simultaneously for all $B \in \B_{(0,+\infty)}$ with $\nu(\partial B) = 0$ and $\nu(B) \neq 0$ we have,
\begin{equation} \label{eq:convmain3}
\lim_{t \rightarrow +\infty} \frac{|N^{(x)}_t(B)|}{\E_x(|N_t(B)|)} = D_\infty^{(x)} \,,
\end{equation} 
The above convergence also holds in $L^1$ for any fixed $B \in \B_{(0,+\infty)}$ with $\nu(B)\neq 0$.
\end{corollary}
\noindent This is a slightly stronger version of~\eqref{eq:growth}.

Lastly, when the dynamics does not die out, we may define for all $t \geq 0$ the {\em empirical distribution} of particles $\nu^{(x)}_t$ via:
\begin{equation}
\nu^{(x)}_t(B):= \frac{|N_t^{(x)}(B)|}{|N_t^{(x)}|} \ \ , \quad  B \in \B_{(0,+\infty)} \,.
\end{equation}
Writing $\nu^{(x)}_t(B) = \big(|N^{(x)}(B)|/\E_x(|N_t|)\big) \cdot \big(\E_x(|N_t|)/|N_t^{(x)}|\big)$ and using~\eqref{eq:d1} and~\eqref{eq:convmain2}, this immediately gives:
\begin{corollary}
For all $x > 0$ we have $\nu^{(x)}_t \Longrightarrow \nu$ as $t \to \infty$ a.e. on $\{|N_t^{(x)}| > 0 \text{ for all }t \geq 0\}$. 
\end{corollary}
\noindent Combining the three corollaries above, we see that in the super-critical regime, there is a positive probability for survival, in which case $|N^{(x)}_t(B)|$ grows like its  expectation for all $B \in \B_{(0,+\infty)}$ with $\nu(\partial B) = 0$ and the empirical distribution $\nu^{(x)}_t$ converges to the quasi-stationary distribution $\nu$ associated with $X$.

\subsection{Motivation and related work}

The $(c,r,\mu)$-dynamics was first introduced by Kesten in his paper \cite{kesten1978} from 1978, arguing that ``it was originally thought that this would be useful for Dr. M. Bramson's thesis \cite{Bram77}, which obtains very precise results for the position of the particle furthest to the right if no absorption on $(-\infty,0)$ takes place''. Since then this model has been studied as a particular example belonging to the class of \textit{branching particle systems with selection}, which are usually of interest for their applications to genetics and population dynamics. 

To understand the latter connection, shift all particles by $+ct$ at time $t$ and then view each particle as an individual whose (shifted) position represents fitness or measure of adaptation to the environment. Then the motion, which is now a standard Brownian motion, represents fitness evolution by mutation, while absorption, which now takes place at a barrier moving at speed $c$, represents the selection effect of removing all individuals whose fitness is too low.

This model is also of importance for its link with the F-KPP equation: indeed, in \cite{harris2006v1} the authors show that the F-KPP travelling wave equation on $\R_{\geq 0}$
$$
\left\{\begin{array}{l}\frac{1}{2}f'' - cf'+r (f^2 - f)=0 \hspace{0.3cm}\text{ on }(0,\infty)\\
f(0^+)=1\\
f(\infty)=0,
\end{array}
\right.
$$ admits a solution if and only if the $(c,r,\mu)$-dynamics with $\mu = \delta_{2}$ (i.e. dyadic branching) survives with positive probability and furthermore that, in this case, the (unique) solution is given by $f(x) = P_x( N_t > 0 \text{ for all }t \geq 0)$.

Branching Brownian motion with drift and absorption is a particular instance of the more general class of {\em branching diffusions}, whereby the motion of particles is that of a general diffusion $X$ on some domain $\cD \subset \R^d$ with generator $\cL$ and with branching according to a fixed law $\mu$, occurring at a rate $r: \cD \to [0,\infty)$, which is allowed to depend on the position of the particle (in general, one may also have the branching law depend on the position of the particle). Such a process can also be viewed as a multi-type branching processes with a general (infinite) type-space. However, unlike in the case of a finite type-space, where the limiting behavior is fully understood~\cite{kesten1966limit}, here a general limit theory is so far restricted to various sub-classes of branching diffusions satisfying additional assumptions.

Notable among such general results are the works of Asmussen and Hering~\cite{asmussen1976,asmussen1977strong} and more recently that of Engl\"{a}nder, Harris and Kyprianou~\cite{englander2010} (which was motivated by earlier works on superprocesses~\cite{EngTur2002,englander2006law}, see also~\cite{englander2014spatial}).
In both cases, the additional assumptions imposed on the process come in the form of regularity and spectral properties of $\cL+r(\mu_1-1)$ -- the generator of the so-called expectation semi-group associated with the dynamics. While the spectral assumptions in~\cite{englander2010} are less restrictive, a key condition present in both lines of work is that the operator $\cL + r(\mu_1-1)-\lambda_c$ is {\em product-critical}, where $\lambda_c$ is the generalized principal eigenvalue of $\cL+r(\mu_1-1)$. Essentially, product-criticality means that the right and left eigenvectors $h$ and $\nu$ (both unique up to constant multiples) corresponding to $\lambda_c$ satisfy $\nu(h)<+\infty$ (see Chapter 4 in~\cite{pinsky1995positive} for further details).

The usefulness of this assumption comes from the fact that if $\cL +r(\mu_1-1)-\lambda_c$ is product-critical then the measure $\nu^h(dx):=h(x)\nu(dx)$ (normalized to satisfy $\nu^h(1)=1$) is, for any $x \in \mathcal{D}$, the unique stationary distribution for the process $X^{(x)}$ under the {\em h-transformed} measure $P^h_x$, defined via
\begin{equation} \label{eq:htransform}
\frac{dP^h_x}{dP}\Bigg|_{\F^{(x)}_t}= M^{(x)}_t \,,
\end{equation} 
with $M^{(x)}$ and $\F^{(x)}$ defined as in~\eqref{eq:MDef}.  By means of the many-to-one lemma (see Lemma \ref{lema:mt1} below), one can then obtain convergence statements for the branching diffusion using the ergodicity of the single-particle motion under the $h$-transform.  Unfortunately, our system is not product-critical as in our case we have  $\lambda_c = r(\mu_1-1) - \lambda$, with $h$ and $\nu$ given by \eqref{eq:hnu} and therefore $\nu(h) = \infty$ 
(alternatively, the $h$-transform of $X^{(x)}$ is a 3-dimensional Bessel process and hence does not admit a stationary distribution).

Limit theorems have been derived in other related models of branching dynamics, such as (the already mentioned) superprocesses (see also~\cite{Chen2015,ChenRenWang2008,LiuRenSong2013}), branching symmetric Hunt processes~\cite{chen2016, ChenShio2007} and general branching Markov processes (e.g. the first part of~\cite{asmussen1976}). In all these cases the presiding assumption is almost always (some form of) product-criticality (although there are exceptions, c.f.~\cite{englander2009law}). Beyond product-criticality and aside from a few ad-hoc examples (e.g.~\cite{watanabe1967}), the only general limit theory for branching diffusions is that in~\cite{JonckSag}. In this recent work, the authors apply second moment arguments to study the convergence in~\eqref{eq:convmain2}, albeit in $L^2$. They show that for a large class of branching diffusions the convergence in~\eqref{eq:convmain2} holds in $L^2$ if and only if the martingale $D^{(x)}$ is bounded in $L^2$. Now, although their theorem does apply in our case, it falls short of implying Kesten's Theorem, because $L^2$-boundedness of $D^{(x)}$ only holds when $r(\mu_1-1)> c^2$ which is more restrictive then our assumption $r(\mu-1)>c^2/2$ and because the convergence is in $L^2$ and hence in probability, but not almost surely.

There is an obvious connection between the problem at hand and the study of high values of ``regular'' branching Brownian motion (i.e., no drift or absorption and $x=0$). Indeed, without absorption, the empirical measures $|N^{(0)}_t(\cdot)|$ identifies with the point process of particles heights for the regular process, shifted by $-ct$ at time $t$. We therefore expect results analogous to those in Proposition~\ref{theo:main2} and Theorem~\ref{theo:main3} to hold in this case, albeit with $h$ and $\nu$ being the exponential function and exponential measure on all $\R$, respectively.
Moreover, the additive martingale will be defined in the same way (using the new $h$) and in~\eqref{eq:convmain2} the sets $B$ will be assumed to be bounded. Similar results, albeit in law, have been derived in the context of the closely related discrete Gaussian free field~\cite{biskup2016intermediate}. 

It is worth mentioning that the additive martingale bears close resemblance to the so-called {\em derivative martingale} introduced by Lalley and Sallke in~\cite{lalley1987conditional} to describe the limiting law for the centered maximum of regular branching Brownian motion.
This martingale is defined as in~\eqref{def:d} with $r(\mu_1-1)-c^2/2 = 0$ (corresponding to the critical case) albeit with a negative sign in front of the sum and, more importantly, without the absorption in the underlying process $N^{(0)}$. As such and unlike the critical additive martingale, the derivative martingale does converge to a non-trivial limit. More recently, it was shown in~\cite{ABBS2013,ABK_E} that the limit of the derivative martingale is also the random constant multiplying the intensity measure of the Cox process which describes the limiting extremal process (i.e. the point process which records all ``nearly maximal'' heights). Thus, in both cases, a similar martingale limit acts as an overall (random) scale factor for the limiting measure.

Lastly, although we focus here on the super-critical case for the $(c,r,\mu)$-dynamics, we note that not less attention is given in the literature to the critical and sub-critical regimes of this process. Results for these regimes include, to name a few, the study of the asymptotic decay of the survival probability as a function of time $t$, initial position $x$ and ``distance'' to criticality $r(\mu_1-1)-c^2/2$ (see, e.g.~\cite{harris2007survival,berestycki2011survival,berestycki2014critical}), the total number of born or absorbed particles (\cite{maillard2013number}, see also~\cite{aidekon2010tail,aidekon2013precise}) and scaling limits in the near-critical regime~\cite{berestycki2013genealogy}.

\subsection{A word about the proof}
Let us conclude this section with a brief overview of the proof of Theorem~\ref{theo:main3} as it demonstrates most of the key ideas in the proof of Proposition~\ref{theo:main2} as well. As in~\cite{JonckSag}, the proof is based on a second moment argument. Fixing $B \in \B_{(0,+\infty)}$, we wish to show that 
\begin{equation}
\label{eq:2ndMomBound}
\E_x (|N_t(B)|^2) \leq C \big(\E_x (|N_t|)\big)^2 \,,
\end{equation}
for some $C > 0$ and all $t \geq 0$ and $x > 0$. Once this is established, we can use the branching structure of the process and condition on $\F_s$ to express $|N^{(x)}_{t+s}(\cdot)|$ as a sum of (conditionally) independent random variables $|N^{(u)}_t(\cdot)|$, one for each $u \in N^{(x)}_s$. Taking expectation and using~\eqref{eq:2ndMomBound}, we can then get
\begin{equation}
\E_x \left( \frac{|N_{t+s}(B)| - \E_x \big(|N_{t+s}(B)|\,\big|\, \F_s \big)}
	{\E_x(|N_{t+s}|)} \right)^2 \leq 
C \frac{\E_x \Big(\sum_{u \in N^{(x)}_s} \big(\E_u (|N_t|)\big)^2 \Big)}
{\big(\E_x \sum_{u \in N^{(x)}_s} \E_u (|N_t|) \big)^2} \,.
\end{equation}
Now an explicit calculation, using the many-to-one lemma (Lemma~\ref{lema:mt1}), shows that the right hand side goes to $0$ exponentially fast when $s \to \infty$ uniformly in all $t \geq t_0(s)$. This implies that as $s \to \infty$ (uniformly in all $t \geq t_0(s)$), the random variable $|N^{(x)}_{t+s}(B)|/\E_x |N_{t+s}|$ gets arbitrarily close in $L^2$ to its conditional expectation given $\F_s$. But since the latter can be shown to converge to $\nu(B) \cdot D_\infty^{(x)}$ in $L^2$ as $t \to \infty$ followed by $s \to \infty$, this shows~\eqref{eq:convmain2} in $L^2$.

The trouble with this argument, is that it only works when $r(\mu_1 - 1) > c^2$, as only in this case do we have~\eqref{eq:2ndMomBound}. This can be easily verified, by explicitly computing both sides of~\eqref{eq:2ndMomBound} using the many-to-one and many-to-two lemmas (Lemmas~\ref{lema:mt1} and~\ref{lema:mt2}). In order to handle also the range of parameters $r(\mu_1 - 1) \in (c^2/2, c^2]$, we introduce next a truncated version $\tilde{N}^{(x), M}$ 
 of the process $N^{(x)}$, which is obtained from $N^{(x)}$ by keeping at any time $t \geq 0$ only those particles whose trajectory stayed below the curve $s \mapsto M(1+s^{3/4})$ for all $s \in [0,t]$. We then show that this truncation is strong enough to guarantee that for any $M$ the process $\tilde{N}^{(x), M}$ will satisfy~\eqref{eq:2ndMomBound} (with $C$ depending on $M$), but also weak enough, so that the $L^1$ distance between $\E_x (|\tilde N_t^M(B)|)/\E_x (|N_t|)$ and
$\E_x (|N_t(B)|)/\E_x (|N_t|)$ tends to $0$ as $M \to \infty$ uniformly in $t$ large enough. Combining the last two assertions shows that~\eqref{eq:convmain2} holds in $L^1$.

To go from $L^1$ convergence to an almost-sure one, we first pick a sequence of times $(t_k)_{k \geq 1}$ tending to infinity fast enough so that the $L^1$ distance from the limit in~\eqref{eq:convmain2} is summable in $k$, but slow enough so that the gaps $t_{k+1} - t_k$ tend to $0$ as $k \to \infty$. This is always possible, thanks to the underlying branching structure, which guarantees that the $L^1$ convergence in~\eqref{eq:convmain2} is at least stretched-exponentially fast. We then use the summability in $k$ together with the Borel-Cantelli Lemma to show that~\eqref{eq:convmain2} holds almost-surely along the sequence $(t_k)_{k \geq 1}$. At the same time, the fact that the gaps vanish in the limit, allows us to show that with probability $1$,
\begin{equation}
\lim_{k \to \infty} \sup_{s \in [t_k, t_{k+1}]} 
	\left| \frac{|N^{(x)}_{s}(B)|}{\E_x (|N_s|)} - \frac{|N^{(x)}_{t_k}(B)|}{\E_x (|N_{t_k}|)} \right| = 0 \,,
\end{equation}
By combining the last two assertions, the desired almost-sure convergence follows.

The remainder of the article is organized as follows. In Section~\ref{sec:pf1} we prove Proposition~\ref{theo:main2}. We begin by recalling in Subsection~\ref{sec:mtf} the many-to-few lemmas (Lemmas~\ref{lema:mt1} and~\ref{lema:mt2}), which will be used repeatedly in the sequel. As in the proof of Theorem~\ref{theo:main3}, going beyond the second moment regime requires a truncated version of the additive martingale, and the latter is introduced in Subsection~\ref{ss:TruncD}. The remaining subsections include the rest of the proof of the theorem. Section~\ref{sec:pf2} is devoted to showing the $L^1$ convergence in Theorem~\ref{theo:main3}. It begins with Subsection~\ref{ss:Sharp}, where sharp asymptotics for $P_x(X_t \in B)$ are derived and continues with Subsection~\ref{sec:part1}, where the truncation of $N^{(x)}$ is defined. The proof is completed in Subsections~\ref{ss:concentration} and~\ref{sec:conc}. Finally, Section~\ref{sec:pf3} includes the proof for the almost-sure convergence in Theorem~\ref{theo:main3}, first along a particular sequence $(t_k)_{k \geq 1}$ (Subsection~\ref{ss:subseq}) and then as $t \to \infty$ (Subsection~\ref{ss:full}).

\section{Proof of Proposition \ref{theo:main2}}\label{sec:pf1}
In this section we give the proof of Proposition \ref{theo:main2}. First, let us notice that, since we know already that $D^{(x)}$ converges almost surely, in order to derive its $L^1$-convergence it will suffice to show that it is uniformly integrable, i.e. that
\begin{equation}
\label{eq:aui}
\lim_{K \rightarrow +\infty}\left[ \sup_{t \geq 0} \E_x( |D_t| \mathbbm{1}_{\{|D_t| > K\}})\right] = 0.
\end{equation} 

It is shown in \cite{JonckSag} that $D^{(x)}$ is bounded in $L^2$ if and only if $r(\mu_1-1)>2\lambda$ so that, in particular, for this values of $r,\mu_1$ and $\lambda$ we already have the uniform integrability. However, for $r(\mu_1-1) \in (\lambda,2\lambda]$ the uniform integrability does not follow from the approach in \cite{JonckSag} and will require a new method, one which is based on truncations of the additive martingale. 
The truncated process will turn out to be uniformly integrable (bounded in $L^2$, in fact) but still asymptotically equivalent in $L^1$ to the entire martingale $D^{(x)}$. From this, the desired $L^1$-convergence will follow.

\subsection{The many-to-few lemmas}\label{sec:mtf}

A key ingredient in the proofs of both the current proposition and Theorem~\ref{theo:main3} is a precise computation of certain first and second moments associated with the process $A=(A_t)_{t \geq 0}$. Such computations can be done easily with the help of the so-called {\em many-to-few} lemmas, which we proceed to recall. For simplicity, we will state only a simplified version of the many-to-one and many-to-two lemmas, which are all we need. For the many-to-few lemma in its full generality (and its proof) we refer to \cite{harris2015}. 

First, we notice that for any $u \in A_t$, the path $\overline{u}_t=(u_s)_{s \in [0,t]}$ is a continuous function. Therefore, it makes sense to consider for each $t > 0$ the space $C[0,t]$ of continuous functions $g:[0,t] \rightarrow \R$ endowed with a measurable space structure by considering the $\sigma$-algebra of Borel sets induced by the supremum distance on $C[0,t]$. Now, let us state the many-to-one lemma.

\begin{lemma}[Many-to-one Lemma] \label{lema:mt1} Given $t > 0$ and a measurable function $f:C[0,t] \rightarrow \R_{\geq 0}$, for every $x > 0$ we have
	\begin{equation}
	\label{eq:mt1}
	\E_x \left( \sum_{u \in A_t} f\left( \overline{u}_t\right)\right) = e^{r(\mu_1 - 1)t}\E_x \left( f\left( \overline{X}_t\right) \right).
	\end{equation}
\end{lemma}

Next, we state the many-to-two lemma, which we use to compute correlations between pairs of particles. Before we can do so, however, we must introduce the notion of the $2$-spine process associated with our branching dynamics: 
\begin{definition} Consider the following branching dynamics on $\R_{\geq 0}$:
	\begin{enumerate}[wide, labelindent=0pt]
		\item [i.] The dynamics starts with 2 particles, both located initially  at some $x > 0$, whose positions evolve together randomly, i.e. describing the same random trajectory, according to $\mathcal{L}$. 
		\item [ii.] These particles wait for a random exponential time $E$ of parameter $(\mu_2-\mu_1)r$, independently of their joint trajectory, and then split at their current position, each of them then evolving independently according to $\mathcal{L}$. 
	\end{enumerate} 
	Now, for $i=1,2$, let $S^{(i)}=(S^{(i)}_t)_{t \geq 0}$ be the process which indicates the position of the $i$-th particle. We call the pair $(S^{(1)},S^{(2)})$ the $2$-spine process associated with the $(\mu,r,\mathcal{L})$-branching dynamics (or just $2$-spine for short) and $E$ its splitting time. 
\end{definition}

The many-to-two lemma then goes as follows.

\begin{lemma}[Many-to-two Lemma] \label{lema:mt2} Given $t > 0$ and measurable functions $f_1,f_2:C[0,t] \rightarrow \R_{\geq 0}$, for every $x > 0$ we have
	$$
	\E_x \left( \sum_{u,v \in A_t} f_1\left( \overline{u}_t\right) f_2\left(\overline{v}_t\right) \right) = e^{2r(\mu_1 - 1)t}\E_x \left( e^{r[\text{Var}(m)+(\mu_1-1)^2] (E \wedge t)}f_1(\overline{S}^{(1)}_t)f_2(\overline{S}^{(2)}_t)\right),
	$$ where $(S^{(1)},S^{(2)})$ is a $2$-spine associated with $(\mu,r,\mathcal{L})$ and $E$ denotes its splitting time.
\end{lemma}

\subsection{Truncation of $D_t^{(x)}$} 
\label{ss:TruncD}
Given $t,M > 0$ let us first set $J_t^{(M)}:= \left[0,M(1+ t^{\frac{3}{4}})\right)$ and then define, for any $x,M > 0$, the truncated process $\tilde{D}^{(x),M}=(\tilde{D}^{(x),M})$ by the formula
$$
\tilde{D}^{(x),M}_t:= \frac{1}{h(x)}\sum_{u \in \tilde{N}^{(x),M}_t} h(u_t)e^{-(r(\mu_1-1)-\lambda)t} \,,
$$ where 
\begin{equation} \label{eq:defn}
\tilde{N}^{(x),M}_t:=\{ u \in A_t^{(x)} : u_s \in J^{(M)}_s \text{ for all }s \in [0,t]\}.
\end{equation} The key properties of the truncated process $\tilde{D}^{(x),M}$ are contained in the two propositions below.

\begin{proposition}\label{prop:d1} For each $x> 0$ there exist constants $C,M_0 >0 $ (depending only on $x$ and $c$) such that for any $M \geq M_0$ one has
	\begin{equation} \label{eq:trunD}
	\sup_{t \geq 0} \E_x(|D_t - \tilde{D}_t^{M}|) \leq e^{-CM^2}.
	\end{equation}
\end{proposition}

\begin{proposition}\label{prop:d2} For any $x,M > 0$ one has that
	$$
	\sup_{t \geq 0} \E_x (|\tilde{D}^{M}_t|^2)<+\infty.
	$$
\end{proposition}

From these two properties it is straightforward to deduce the uniform integrability of $D^{(x)}$. Indeed, it follows from Proposition \ref{prop:d2} that $\tilde{D}^{(x),M}$ is uniformly integrable for each $M > 0$ so that it is now an simple exercise using \eqref{eq:trunD} to see that $D^{(x)}$ must be also. Therefore, in order to obtain the first statement of Proposition \ref{theo:main2}, it will suffice to show Propositions \ref{prop:d1} and \ref{prop:d2} above.

\subsection{Proof of Proposition \ref{prop:d1}}

Let us note that, by the many-to-one lemma (Lemma \ref{lema:mt1}), \eqref{eq:trunD} is equivalent to showing that there exists $C > 0$ such that for all $M$ sufficiently large
	\begin{equation} \label{eq:ht}
    \sup_{t \geq 0}	P^h_x( \exists \,s \in [0,t] \text{ such that } X_s \notin J_s^{(M)}) \leq e^{-CM^2},
	\end{equation} where $P^h$ is the $h$-transform of $X$ given by \eqref{eq:htransform}. Therefore, it will suffice to show \eqref{eq:ht}. To do this, we note that if $H^{(x),M}:=\inf\{ s \geq 0 : X_s^{(x)} \notin J_s^{(M)}\}$ then by the strong Markov property for $H^{(x),M}$ (under the measure $P$) we have the bound
	\begin{align*}
	P^h_x( \exists \,s \in [0,t] \text{ such that } X_s \notin J_s^{(M)}) &= P^h_x( H^{M} \leq t)\\
	& = \frac{e^{\lambda t}}{h(x)}\int_0^{t} \E_{M(1+s^{\frac{3}{4}})}(h(X_{t-s})) P_x(H^M \in ds) 
	\\ & \leq \sum_{k=0}^{\lfloor t \rfloor} h(M(1+(k+1)^{\frac{3}{4}})e^{\lambda (k+1)}P_x( H^M \in [k,k+1]), 
	\end{align*} where in the last line we have used that $\E_y(M_s)=1$ holds for every $y,s> 0$.

Now, it follows from  the Reflection Principle and standard Gaussian estimates that for $M$ sufficiently large (depending only on $x$ and $c$) 
	\begin{align*}
	P_x(H^M \in [k,k+1]) &\leq P\left( \sup_{s \in [0,k+1]} B_s \geq M(1+k^{\frac{3}{4}})+ck -x\right)
	\\
	& = 2 P\left( B_{k+1} \geq M(1+k^{\frac{3}{4}})+ck -x)\right)
	\\
	& \leq \frac{C_1}{h(M(1+(k+1)^{\frac{3}{4}}))}e^{-C_2 M^2(1+\sqrt{k})-\lambda (k+1)} \,,
	\end{align*} for some constants $C_1,C_2> 0$ depending only on $x$ and $c$. In particular, for all $M$ sufficiently large 
	$$
	\sup_{t \geq 0} P^h_x( \exists \,s \in [0,t] \text{ such that } X_s \notin J_s^{(M)}) \leq C_3 e^{-C_2 M^2} \,,
	$$ for some $C_3=C_3(x,c)>0$ and so \eqref{eq:ht} now follows.

\subsection{Proof of Proposition \ref{prop:d2}}

Observe that by many-to-two lemma (Lemma \ref{lema:mt2}) we have 
	\begin{align*}
	\E_x(|\tilde{D}_t^M|^2) &= \frac{e^{-2(r(\mu_1-1)-\lambda)t}}{h^2(x)}\E_x \left( \sum_{u,v \in A_t} h(u_t)h(v_t)\mathbbm{1}_{\{\overline{u}_t,\overline{v}_t \in T^{(M)}_t\}}\right) \\
	& = \frac{e^{2\lambda t}}{h^2(x)}\E_x \left( h(S^{(1)}_t)h(S^{(2)}_t) \mathbbm{1}_{\{\overline{S}^{(1)}_t,\overline{S}^{(2)}_t \in T^{(M)}_t\}}e^{r[\text{Var}(m)+(\mu_1-1)^2](E\wedge t)}\right),
	\end{align*} where 
	$$
	T^{(M)}_t:=\{ g \in C[0,t] : g(s) \in J^{(M)}_s \text{ for all }s \in [0,t]\}.
	$$ 
	By separating in cases depending on whether $E > t$ or not and using the independence of $E$ from the motion of the $2$-spine, a simple calculation yields
	$$
	\E_x(|\tilde{D}^M_t|^2) = (1)_{t} + (2)_{ t},
	$$ with
	\begin{equation} \label{eq:boundd1}
	(1)_{t} := \frac{e^{-(r(\mu_1-1)-\lambda)t}}{h(x)}\E^h_x(h(X_t)\mathbbm{1}_{\{\overline{X}_t \in T^{(M)}_t\}}) \leq \frac{e^{-(r(\mu_1-1)-\lambda)t}}{h(x)} h(M(1+t^{\frac{3}{4}}))
	\end{equation} and 
	$$
	(2)_{t}: = \frac{(\mu_2-\mu_1)r}{h^2(x)}e^{2\lambda t} \int_0^t \E_x\left(h(X^{(1),s}_t)h(X^{(2),s}_t) \mathbbm{1}_{\{\overline{X}^{(1),s}_t,\overline{X}^{(2),s}_t \in T^{(M)}_t\}}\right)e^{-r(\mu_1-1)s}ds,
	$$ where $X^{(1),s}$ and $X^{(2),s}$ are two coupled copies of the Markov process $X$ which coincide \mbox{until time $s$} and then evolve independently after $s$. 

	Now, notice that
	$$
	(2)_t \leq \frac{(\mu_2-\mu_1)r}{h^2(x)}e^{2\lambda t} \int_0^t \E_x\left(h(X^{(1),s}_t)h(X^{(2),s}_t) \mathbbm{1}_{\{X^{(1),s}_s <M(1+s^{\frac{3}{4}})\}}\right)e^{-r(\mu_1-1)s}ds \,,
	$$ so that, by conditioning on $X^{(1),s}_s$, we obtain that
	\begin{align*} \label{eq:2thd}
	(2)_{t} &\leq \frac{(\mu_2-\mu_1)r}{h(x)}\int_0^t \E_x\left(\mathbbm{1}_{\{X_s \leq M(1+s^{\frac{3}{4}})\}}M_s h(X_s)\E^2_{X_s}(M_{t-s})\right)e^{-(r(\mu_1-1)-\lambda)s}ds\\
	& \leq \frac{(\mu_2-\mu_1)r}{h(x)}\int_0^t h(M(1+{s}^{\frac{3}{4}}))e^{-(r(\mu_1-1)-\lambda)s}ds \\
	& \leq \frac{(\mu_2-\mu_1)r}{h(x)}\int_0^\infty h(M(1+{s}^{\frac{3}{4}}))e^{-(r(\mu_1-1)-\lambda)s}ds<+\infty \,,
	\end{align*} from where, together with \eqref{eq:boundd1}, the result now follows.	

\subsection{Strict positivity of $D^{(x)}_\infty$ in the event of survival}

To conclude the proof of Proposition~\ref{theo:main2}, it only remains to show that
\begin{equation}
\label{eq:dpos}
P_x( D_\infty > 0 \,|\, |N_t| > 0 \text{ for all }t) = 1.
\end{equation} Notice that since
\begin{equation} \label{eq:dpos2}
\{|N_t| = 0 \text{ for some }t\} \subseteq \{D^{(x)}_\infty = 0\},
\end{equation} in order to obtain \eqref{eq:dpos} it will be enough to show that both events in \eqref{eq:dpos2} have the same probability. For this purpose, we follow the approach in \cite{harris2006v1}. Let us define $\sigma:(0,+\infty) \rightarrow [0,1]$ by the formula
$$
\sigma(x):=P_x(D_\infty = 0).
$$ Observe that $\sigma$ is monotone decreasing. Indeed, since for any pair $x \leq y$ one has that $N^{(x)} \preceq N^{(y)}$, i.e. there exists a coupling of these processes such that 
$$
|N^{(x)}_t((a,+\infty))| \leq  
|N^{(y)}_t((a,+\infty))| \,,
$$ for every $a > 0$, $t \geq 0$. Using this coupling and the monotonicity of $h$ one can construct for any $x \leq y$ versions of $D^{(x)}$ and $D^{(y)}$ such that $h(x)D^{(x)}_t \leq h(y)D^{(y)}_t$ for all $t \geq 0$. In particular, by taking $t \rightarrow +\infty$ on this inequality we see that $h(x)D^{(x)}_\infty \leq h(y)D^{(y)}_\infty$, which implies the monotonicity of $\sigma$.

On the other hand, an easy computation using the branching property shows that
$$
D^{(x)}_{\infty} = \frac{e^{-(r(\mu_1-1)-\lambda)}}{h(x)} \sum_{u \in N^{(x)}_1} h(u_1) D^{(u_1)}_\infty \,,
$$ where the random variables $\{D^{(u_1)}_\infty : u \in N^{(x)}_1\}$ are all independent. In particular, we obtain that $D^{(x)}_\infty=0$ if and only if $D^{(u_1)}_\infty=0$ for all $u \in N^{(x)}_1$, which yields
\begin{equation} \label{eq:fp}
\sigma(x)= \E_x\left( \prod_{u \in N_1} \sigma(u_1)\right) \,,
\end{equation} with the convention that $\prod_{u \in \emptyset} \sigma(u_1) = 1$, used when $|N_1^{(x)}|=0$. 

Now, if we denote by $\bar{N}=(\bar{N}_t)_{t \geq 0}$ the branching dynamics (starting at $0$)  associated with a Brownian motion with drift $-c$ but \textit{without} killing at $0$, then it is not hard to construct a coupling between $\bar{N}$ and $\{N^{(x)}:x >0\}$ such that the limits
$$
\lim_{x \rightarrow +\infty} |N_1^{(x)}((0,a))|=0\hspace{1cm}\text{ and }\hspace{1cm}
\lim_{x \rightarrow +\infty} |N_1^{(x)}| = |\bar{N}_1|
$$ hold almost surely for any $a > 0$. Using this coupling, the monotonicity of $\sigma$ yields
$$
\sigma(\infty):=\lim_{x \rightarrow +\infty} \sigma(x) = \E \left(\sigma(\infty)^{|\bar{N}_1|}\right),
$$ from where we conclude that $\sigma(\infty)$ must be either $0$ or $1$. But since we have already shown that $D^{(x)}_{t}$ converges in $L^1$ to $D^{(x)}_\infty$, we have that $\E_x(D_\infty)=1$ and therefore that $\sigma(x)< 1$ for all $x > 0$, so that it must be $\sigma(\infty)=0$. 

Iterating the relation in \eqref{eq:fp} yields that
\begin{equation} \label{eq:dom}
\sigma(x)=\E_x\left( \prod_{u \in N_n} \sigma(u_n)\right) \leq P_x(|N_n|=0)+ \E_x\left( \sigma\left(\max_{u \in N_n} u_n\right) \mathbbm{1}_{\{|N_n|>0\}}\right) \,,
\end{equation} for every $n \in \N$. But, since one has $\lim_{n \rightarrow +\infty} [\max_{u \in N_n^{(x)}} u_n]=+\infty$ in the event of survival by \cite[Lemma 2]{harris2006v1}, by taking the limit as $n \rightarrow +\infty$ in \eqref{eq:dom} and using the bounded convergence theorem, we conclude that
$$
\sigma(x) \leq P_x( |N_t|= 0 \text{ for some }t) +\sigma(\infty)= P_x( |N_t|= 0 \text{ for some }t).
$$ Since the reverse inequality is immediate by \eqref{eq:dpos2}, the result now follows.
 
\section{Proof of Theorem \ref{theo:main3} - $L^1$-convergence}\label{sec:pf2}

We shall first show that~\eqref{eq:convmain2} holds in $L^1$. As in the case of Proposition~\ref{theo:main2}, the results in \cite{JonckSag} can be used to show that~\eqref{eq:convmain2} holds in $L^2$ if and only if $r(\mu_1-1) > 2\lambda$, but their approach cannot be used directly to show $L^1$-convergence in the region $r(\mu_1-1) \in (\lambda,2\lambda]$ where the $L^2$-norm is in fact exploding, i.e.
$$
\lim_{t \rightarrow +\infty} \frac{\E_x(|N_t|^2)}{\E_x^2(|N_t|)}=+\infty.
$$ Thus, we must resort to truncations once again to obtain the desired result.

\subsection{Sharp asymptotics for $P_x(X_t \in B)$} 
\label{ss:Sharp}

The first step towards proving Theorem \ref{theo:main3} will be to obtain suitable bounds for the error term in the asymptotics shown in \eqref{eq:asymp}. These are contained in the following lemma.

\begin{lemma}\label{lemma:asymp} For any $x,t>0$ we have that
	\begin{equation}\label{eq:asymp2}
	P_x(X_t > 0) = h(x)t^{-\frac{3}{2}}e^{-\lambda t}(1+\varepsilon(x,t)) \,,
	\end{equation} where the error term $\varepsilon(x,t)$ satisfies the bounds
	\begin{equation} \label{eq:asymp3}
	e^{-\frac{x^2}{2t}}\left(1-\frac{3}{2}t^{-1}\right) \leq 1+\varepsilon(x,t) \leq 1.
	\end{equation}
	Furthermore, for any $B \in \B_{(0,+\infty)}$ we have that
	\begin{equation}
	\label{eq:asympB}
	P_x(X_t \in B) = h(x)t^{-\frac{3}{2}}e^{-\lambda t}(\nu(B)+\varepsilon_B(x,t)) \,,
	\end{equation} where the error term $\varepsilon_B(x,t)$ satisfies the bound
	\begin{equation} \label{eq:cotaeb}
	|\varepsilon_B(x,t)|\leq \left[ \frac{C}{t}(x+1)^2\right] \wedge 2 \,,
	\end{equation} for some constant $C > 0$ depending only on $c$. 
	In particular, we have as $t \to \infty$,
	\begin{equation}
	\label{eq:NAsy}
	\E_x(|N_t(B)|) \sim e^{(r(\mu_1-1) - \lambda)t} t^{-3/2} h(x) \nu(B) \,.
	\end{equation}
\end{lemma} 

\begin{proof} It is well-known that $H^{(x)}_0:=\inf\{s \geq 0 : X^{(x)}_s = 0\}$ has inverse Gaussian distribution with (single) parameter $x$. Thus, we see that
	$$
	P_x(X_t > 0)=P_x( H_0 > t)= \int_t^{\infty} \frac{x}{\sqrt{2\pi s^3}}\exp\left(cx-\lambda s - \frac{x^2}{2s}\right)ds \,,
	$$ which yields the bounds
	$$
	e^{-\frac{x^2}{2t}} \frac{xe^{cx}}{\sqrt{2\pi}} \Gamma_t \leq P_x(X_t > 0) \leq \frac{x e^{cx}}{\sqrt{2\pi}} \Gamma_t \,,
	$$ where 
	$$
	\Gamma_t:= \int_t^\infty s^{-\frac{3}{2}}e^{-\lambda s}ds.
	$$ Now, upon observing the simple bound
	$$
	\Gamma_t \leq t^{-\frac{3}{2}} \int_t^\infty e^{-\lambda s}ds = \frac{1}{\lambda}t^{-\frac{3}{2}} e^{-\lambda t}.
	$$ and also that by integration by parts we have
	$$
	\Gamma_t = \frac{1}{\lambda} t^{-\frac{3}{2}} e^{-\lambda t} - \frac{3}{2\lambda} \int_t^\infty s^{-\frac{5}{2}}e^{-\lambda s}ds \geq \frac{1}{\lambda}t^{-\frac{3}{2}}e^{-\lambda t}\left(1 - \frac{3}{2}t^{-1}\right), 
	$$ we conclude \eqref{eq:asymp2}. 

	On the other hand, for any $B \in \mathcal{B}_{(0,+\infty)}$ we can write
	$$
	P_x(X_t \in B) = \E_x( \mathbbm{1}_B(X_t)) = h(x)e^{-\lambda t}\E^h_x\left( \frac{1}{h(X_t)}\mathbbm{1}_B(X_t)\right) = h(x)t^{-\frac{3}{2}}e^{-\lambda t}(\nu(B)+\varepsilon_B(x,t))
	\,,$$ with $\varepsilon_B$ given by
	\begin{equation}
	\label{eq:expb}
	\varepsilon_B(x,t):=t^{\frac{3}{2}}\E^h_x\left( \frac{1}{h(X_t)}\mathbbm{1}_B(X_t)\right) - \nu(B).
	\end{equation}
	Now, from the proof of \cite[Theorem 2]{polak2012} (where an explicit formula for the density of $X_t^{(x)}$ is given) we obtain that
	$$
	t^{\frac{3}{2}}\E^h_x\left( \frac{1}{h(X_t)}\mathbbm{1}_B(X_t)\right) = \int_B 2\lambda ye^{-cy}\left[ \frac{ \exp\{-\frac{(x-y)^2}{2t}\} - \exp\{ -\frac{(x+y)^2}{2t}\}}{\frac{2xy}{t}}\right]dy \,,
	$$ so that
	\begin{equation}
	\label{eq:boundeps}
	|\varepsilon_B(x,t)| \leq \int_0^\infty 2\lambda y e^{-cy} \left| \frac{ \exp\{-\frac{(x-y)^2}{2t}\} - \exp\{ -\frac{(x+y)^2}{2t}\}}{\frac{2xy}{t}} - 1\right| dy.
	\end{equation} Since by the mean-value theorem we have that
	$$
	\left| \frac{ \exp\{-\frac{(x-y)^2}{2t}\} - \exp\{ -\frac{(x+y)^2}{2t}\}}{\frac{2xy}{t}} - 1\right| \leq \frac{(x+y)^2}{2t} \,,
	$$ for all $x,y \geq 0$, plugging this into \eqref{eq:boundeps} yields
	$$
	|\varepsilon_B(x,t)| \leq \frac{C}{t}(x+1)^2 \,,
	$$ for some constant $C > 0$ depending on $c$. On the other hand, since we have $\varepsilon(x,t)=\varepsilon_{(0,+\infty)}(x,t)$, using the expression \eqref{eq:expb} we obtain 
	$$
	|\varepsilon_B(x,t)| \leq |\varepsilon(x,t)| + |1-\nu(B)| \leq 2 \,,
	$$ since $-1 \leq \varepsilon(x,t) \leq 0$ by \eqref{eq:asymp3} and the fact that $1+\varepsilon(x,t)$ must be positive due to \eqref{eq:asymp2}.
The last part of the lemma is immediate from~\eqref{eq:asympB},~\eqref{eq:cotaeb} and the many-to-one lemma (Lemma~\ref{lema:mt1}).
\end{proof}

\subsection{Truncation of $N_t^{(x)}$}\label{sec:part1} Given any $x,M > 0$, $s,t \geq 0$ and $B \in \B_{(0,+\infty)}$ we define the \textit{$s$-shifted $M$-truncation} by the formula
\begin{equation}
\label{eq:NTrunc}
\tilde{N}^{(x),\{M,s\}}_t(B) := \{ u \in A^{(x)}_{t} : u_{z} \in J^{(M)}_{s+z} \text{ for all }z \in [0,t]\,,\,u_{t} \in B\} \,.
\end{equation}
For simplicity, we shall write  $\tilde{N}^{(x),M}_t(B):=\tilde{N}^{(x),\{M,0\}}_t(B)$ when $s=0$, in accordance with \eqref{eq:defn}. The crucial properties of the $M$-truncated process are contained in the two propositions below. These are analogues of Proposition~\ref{prop:d1} and Proposition~\ref{prop:d2} from Section~\ref{sec:pf1}.

\begin{proposition}\label{prop:n1} For any $x > 0$ there exist constants $C,t_0,M_0 > 0$ (depending only on $x$ and $c$) such that for any $B \in \B_{(0,+\infty)}$, $t \geq t_0$ and $M \geq M_0$ one has
	$$
	\E_x\left(\left|\frac{|N_{t}(B)|}{\E_x(|N_{t}|)}-\frac{|\tilde{N}^{M}_t(B)|}{\E_x(|N_{t}|)}\right|\right) \leq e^{-CM^2} \,. 
	$$  
\end{proposition}

\begin{proposition}\label{prop:n2} There exist constants $C,\delta > 0$ depending only on $r,\mu$ and $c$ which satisfy that, for all $s,M$ such that $1 \leq M \leq \delta s^{\frac{1}{4}}$ one has
	\begin{equation}
	\label{eq:n2}
	\E_x(|\tilde{N}^{\{M,s\}}_t|^2) \leq C h(x) e^{\frac{1}{2}(r(\mu_1-1)-\lambda) s} \left(t^{-\frac{3}{2}}e^{(r(\mu_1-1)-\lambda) t}\right)^2.
	\end{equation} for every $x>0$ and all $t$ sufficiently large (depending only on $r,\mu$ and $c$). 	
\end{proposition}

\begin{proof}[Proof of Proposition \ref{prop:n1}] 

By the many-to-one lemma (Lemma \ref{lema:mt1}), it will be enough to show that there exists $C  > 0$ such that for all $t,M$ sufficiently large (depending only on $x$ and $c$),
\begin{equation} \label{eq:nt}
P_x( \exists \,z \in [0,t] \text{ such that } X_z \notin J_z^{(M)}\,|\,X_t > 0) \leq e^{-CM^2} \,,
\end{equation} and proceeding as in the proof of Proposition \ref{prop:d1}, we have that
$$
P_x( \exists \,z \in [0,t] \text{ s.t. } X_z \notin J_z^{(M)}\,,\,X_t > 0) \leq \sum_{k=0}^{\lceil t \rceil-1} P_{M(1+(k+1)^{\frac{3}{4}})}(X_{(t-(k+1)) \vee 0}> 0)P_x( H^M \in [k,k+1]). 
$$ 

Now, for every $k \in \{0,\dots,\lceil t \rceil - 3 \}$, by Lemma \ref{lemma:asymp},
$$
P_{M(1+(k+1)^{\frac{3}{4}})}(X_{(t-(k+1)}> 0) \leq h \big(M(1+(k+1)^{\frac{3}{4}})) \big) (t-(k+1))^{-\frac{3}{2}}e^{-\lambda(t-(k+1)} \,,
$$ while the proof of Proposition \ref{prop:d1} gives that, for all $M$ large enough (depending only on $x$ and $c$),
$$
P_x(H^M \in [k,k+1]) \leq  \frac{C_1}{h(M(1+(k+1)^{\frac{3}{4}}))}e^{-C_2 M^2(1+\sqrt{k})-\lambda (k+1)} \,,
$$ for some constants $C_1,C_2> 0$ depending only on $x$ and $c$ and all $k \in \{0, \dots, \lceil t \rceil -1\}$. Hence, by combining both estimates we see that if $t$ and $M$ are sufficiently large (depending only on $x$ and $c$) then
$$
\sum_{k=0}^{\lfloor t/2 \rfloor} P_{M(1+(k+1)^{\frac{3}{4}})}(X_{t-(k+1)}> 0)P_x( H^M \in [k,k+1]) \leq C_3e^{-C_2 M^2} t^{-\frac{3}{2}}e^{-\lambda t}, 
$$
$$
\sum_{k=\lfloor t/2 \rfloor + 1} ^{\lceil t \rceil - 3} P_{M(1+(k+1)^{\frac{3}{4}})}(X_{t-(k+1)}> 0)P_x( H^M \in [k,k+1]) \leq C_4 e^{-C_2M^2} e^{-\frac{C_2}{2}\sqrt{t/2}}e^{-\lambda t}  
$$ and
$$
P_x \big( H^M \in \big[\lceil t \rceil-2, \lceil t \rceil \big]\big ) \leq  C_5 e^{-C_2M^2} e^{-C_2 \sqrt{t}}e^{-\lambda t} \,,
$$ for some constants $C_3,C_4,C_5 > 0$ depending only on $x$ and $c$. By the asymptotics in Lemma \ref{lemma:asymp}, this is already enough to show \eqref{eq:nt} and thus prove the result.
\end{proof}

\begin{proof}[Proof of Proposition \ref{prop:n2}]

Note that by the many-to-two lemma (Lemma \ref{lema:mt2}) we have 
\begin{align*}
\E_x(|\tilde{N}_t^{\{M,s\}}|^2) &= \E_x \left( \sum_{u,v \in A_t} \mathbbm{1}_{\{ \overline{u}_t \in T^{\{M,s\}}_t\,,\,u_t >0\}}\mathbbm{1}_{\{ \overline{v}_t \in T^{\{M,s\}}_t\,,\,v_t>0\}}\right) \\
& = e^{2r(\mu_1-1)t}\E_x \left(  \mathbbm{1}_{\{\overline{S}^{(1)}_t \in T^{\{M,s\}}_t,S^{(1)}_t>0\}}\mathbbm{1}_{\{\overline{S}^{(2)}_t \in T^{\{M,s\}}_t,S^{(2)}_t>0\}}e^{r[\text{Var}(m)+(\mu_1-1)^2](E\wedge t)}\right),
\end{align*} where 
$$
T^{(x),\{M,s\}}_t:=\{ g \in C[0,t] : g(z) \in J^{(M)}_{s+z} \text{ for all }z \in [0,t]\}.
$$ As in the proof of Proposition~\ref{theo:main2}, by separating in cases depending on whether $E > t$ or not and using the independence of $E$ from the motion of the $2$-spine, we obtain
$$
\E_x(|\tilde{N}^{\{M,s\}}_t|^2) = (1)_{t} + (2)_{t},
$$ where
$$
(1)_{t} := e^{r(\mu_1-1)t}P_x(\overline{X}_t \in T^{\{M,s\}}_t\,,\,X_t > 0) \,,
$$ and
$$
(2)_{t}:= (\mu_2-\mu_1)re^{2r(\mu_1-1)t} \int_0^t P_x\left( \cap_{i=1}^2\{ \overline{X}^{(i),z}_t \in T^{\{M,s\}}_t\,,\,X^{(i),z}_t > 0\}\right)e^{-r(\mu_1-1)z}dz,
$$ where $X^{(1),z}$ and $X^{(2),z}$ are two coupled copies of the Markov process $X$ which coincide \mbox{until time $z$} and then evolve independently after $z$. 

Now, by Lemma \ref{lemma:asymp} we have that
$$
(1)_t\leq e^{r(\mu_1-1)t}P_x(X_t > 0) \leq h(x)t^{-\frac{3}{2}}e^{(r(\mu_1-1)-\lambda)t} \leq h(x) \left[\left(t^{-\frac{3}{2}}e^{(r(\mu_1-1)-\lambda)t}\right)^2 \vee 1\right].
$$ 
On the other hand, observe that
$$
(2)_t \leq (\mu_2-\mu_1)re^{2r(\mu_1-1)t}\int_0^t P_x(X^{(1),z}_z < M(1+(s+z)^{\frac{3}{4}})\,,\,X^{(1),z}_t > 0\,,\, X^{(2),z}_t > 0)e^{-r(\mu_1-1)z}dz\,,
$$ so that, by conditioning on $X^{(1),z}_z$, we obtain
\begin{equation} \label{eq:2th}
(2)_{t} \leq \int_0^t \Psi^{\{M,s\}}_{x,t}(z) dz,
\end{equation}
where
$$
\Psi_{x,t}^{\{M,s\}}(z):=(\mu_2-\mu_1)re^{2r(\mu_1-1)t}  \E_x\left(\mathbbm{1}_{\{X_z \leq M(1+(s+z)^{\frac{3}{4}})\}}P^2_{X_z}(X_{t-z} >0)\right)e^{-r(\mu_1-1)z}.
$$ 

To treat the right-hand side of \eqref{eq:2th} we split the integral into two parts, i.e. for $\alpha \in (0,1)$ we write
$$
\int_0^t \Psi^{\{M,s\}}_{x,t}(z)dz = [a]_{t} + [b]_{t}  \,,
$$ where
$$
[a]_{t} := \int_{\alpha t}^t \Psi^{\{M,s\}}_{x,t}(z)dz \hspace{2cm}\text{ and }\hspace{2cm}
[b]_{t} := \int_0^{\alpha t} \Psi^{\{M,s\}}_{x,t}(z)dz. 
$$ 
For the first term, by the Markov property we have
$$
\E_x(\mathbbm{1}_{\{X_z \leq M(1+(s+z)^{\frac{3}{4}})\}}  P_{X_z}^2(X_{t-z} > 0) ) \leq \E_x(P_{X_z}( X_{t-z}> 0)) = P_x( X_t > 0),
$$ so that by Lemma \ref{lemma:asymp}
\begin{align*}
[a]_{t} &\leq (\mu_2-\mu_1)r h(x)t^{-\frac{3}{2}}e^{(2r(\mu_1-1)-\lambda)t}\int_{\alpha t}^{\infty} e^{-r(\mu_1-1)l}dl\\
& = \frac{\mu_2-\mu_1}{\mu_1-1} h(x)t^{-\frac{3}{2}}e^{(r(\mu_1-1)-\lambda)t} e^{(1-\alpha)r(\mu_1-1)t}\\
& \leq \frac{\mu_2-\mu_1}{\mu_1-1} h(x) \left(t^{-\frac{3}{2}}e^{(r(\mu_1-1)-\lambda)t}\right)^2\,,
\end{align*} if $\alpha$ is chosen sufficiently close to $1$ and $t$ taken large enough (both depending only on $r,\mu$ and $\lambda$).

At the same time, it also follows from Lemma \ref{lemma:asymp} that for any $z \leq \alpha t$ one has
\begin{align}
\Psi^{\{M,s\}}_{x,t}(z) &\leq (\mu_2-\mu_1)r (t-z)^{-3}e^{2(r(\mu_1-1)-\lambda)t} \E_x\left(\mathbbm{1}_{\{X_z \leq M(1+(s+z)^{\frac{3}{4}})\}}h^2(X_z)e^{\lambda z}\right) e^{-(r(\mu_1-1)-\lambda)z} \nonumber \\ 
& \leq (\mu_2-\mu_1)r \left(\frac{t}{t-z}\right)^3 h(x)\left(t^{-\frac{3}{2}}e^{(r(\mu_1-1)-\lambda)t}\right)^2 \E^h_x \left(\mathbbm{1}_{\{X_z \leq M(1+(s+z)^{\frac{3}{4}})\}}h(X_z)\right)e^{-(r(\mu_1-1)-\lambda)z} \nonumber
\\ 
& \leq (\mu_2-\mu_1)r \left(\frac{1}{1-\alpha}\right)^3 h(x) \left(t^{-\frac{3}{2}}e^{(r(\mu_1-1)-\lambda)t}\right)^2 h(M(1+(s+z)^{\frac{3}{4}})) e^{-(r(\mu_1-1)-\lambda)z} \,,\nonumber
\end{align} so that
$$
[b]_t \leq (\mu_2-\mu_1)r \left(\frac{1}{1-\alpha}\right)^3 h(x) \left(\int_0^\infty h(M(1+(s+z)^{\frac{3}{4}}))e^{-(r(\mu_1-1)-\lambda)z}dz\right) \left(t^{-\frac{3}{2}}e^{(r(\mu_1-1)-\lambda)t}\right)^2.
$$ 

Now, if $s,M$ are such that $1 \leq M \leq \delta s^{\frac{1}{4}}$ then it is not hard to show that
$$
h(M(1+(s+z)^{\frac{3}{4}})) \leq \tilde{C}_1 e^{\tilde{C}_2 \delta (s+z)} \,,
$$ for some constants $\tilde{C}_1,\tilde{C}_2 > 0$ depending only on $c$. Thus, by taking $\delta$ sufficiently small so as to guarantee that $\tilde{C}_2 \delta \leq \frac{1}{2}(r(\mu_1-1)-\lambda)$, we conclude that
$$
\int_0^\infty h(M(1+(s+z)^{\frac{3}{4}}))e^{-(r(\mu_1-1)-\lambda)z}dz \leq C e^{\frac{1}{2}(r(\mu_1-1)-\lambda)s} \,,
$$ for some constant $C > 0$ depending only on $r,\mu$ and $\lambda$. Hence, upon recalling the bounds obtained for $[a]_t$ and $[b]_t$, the result now follows.
\end{proof}

\subsection{Concentration of the $M$-truncated process $N^{(x),M}$}
\label{ss:concentration}

As a consequence of Proposition \ref{prop:n2} we have the following concentration result for the $M$-truncated process $N^{(x),M}$. In the following, $(\mathcal{G}^{(x)}_s)_{s \geq 0}$ will denote the filtration generated by the branching dynamics $N^{(x)}$.

\begin{proposition}\label{prop:n3} There exist constants $C,\delta > 0$ depending only on $r,\mu$ and $c$ which satisfy that, for all $s,M$ such that $1 \leq M \leq \delta s^{\frac{1}{4}}$ one has
	$$
	\E_x\left( \left|\frac{\tilde{N}^{M}_{t+s}(B)}{\E_x(N_{t+s})} - \frac{\E_x(\tilde{N}^{M}_{t+s}(B)|\mathcal{G}_s)}{\E_x(N_{t+s})}\right|^2\right) \leq \frac{C}{h(x)} \left(\frac{t+s}{t}\right)^3 e^{-\frac{1}{2}(r(\mu_1-1)-\lambda)s} \,,
	$$ for every $x > 0$, $B \in \B_{(0,+\infty)}$ and all $t \geq 0$ sufficiently large (depending only on $r,\mu$ and $c$).
\end{proposition}
 
\begin{proof} By the branching property of $N^{(x)}$ we have the decomposition 
	$$
	\tilde{N}^{(x),M}_{t+s}(B)-\E_x(\tilde{N}^{M}_{t+s}(B)|\mathcal{G}_s) = \sum_{u \in \tilde{N}^{(x),M}_s} \left(\tilde{N}^{(u_s),\{M,s\}}_t(B) - \E_{u_s}(\tilde{N}^{\{M,s\}}_t(B))\right) \,,
	$$ where all terms appearing in the sum on the right-hand side are independent conditionally on $\mathcal{G}^{(x)}_s$. It follows that 
	$$
	\E_{x}\left(\left|\tilde{N}^{M}_{t+s}(B)-\E_x(\tilde{N}^{M}_{t+s}(B)|\mathcal{G}_s)\right|^2\bigg| \mathcal{G}_s\right) = \sum_{u \in \tilde{N}^{(x),M}_s} \text{Var}_{u_s}(\tilde{N}^{\{M,s\}}_t(B)) \leq \sum_{u \in \tilde{N}^{(x),M}_s} \E_{u_s}(|\tilde{N}^{\{M,s\}}_t|^2). 
	$$ Now, by Proposition \ref{prop:n2} we have that there exist constants $C,\delta > 0$ such that for all $1 \leq M \leq \delta s^\frac{1}{4}$ one has
	$$
	\sum_{u \in \tilde{N}^{(x),M}_s} \E_{u_s}(|\tilde{N}^{\{M,s\}}_t|^2) \leq C e^{\frac{1}{2}(r(\mu_1-1)-\lambda)s} t^{-3}e^{2(r(\mu_1-1)-\lambda)t} \sum_{u \in N^{(x)}_s} h(u_s). 
	$$ for every $t$ sufficiently large (depending only on $r,\mu$ and $\lambda$). On the other hand, since by Lemmas \ref{lema:mt1} and \ref{lemma:asymp} we have that
	$$
	\E_x(N_{t+s})=h(x)(t+s)^{-\frac{3}{2}}e^{(r(\mu_1-1)-\lambda)(t+s)}(1+\varepsilon(x,t+s)) \geq \frac{1}{2}h(x)(t+s)^{-\frac{3}{2}}e^{(r(\mu_1-1)-\lambda)(t+s)} \,,
	$$ for any $t$ sufficiently large so as to have $\inf_{z\geq t}\varepsilon(x,z) \geq -\frac{1}{2}$, we see that
	$$
	\E_x\left( \left|\frac{\tilde{N}^{M}_{t+s}(B)}{\E_x(N_{t+s})} - \frac{\E_x(\tilde{N}^{M}_{t+s}(B)|\mathcal{G}_s)}{\E_x(N_{t+s})}\right|^2\Bigg|\mathcal{G}_s\right) \leq \frac{4C}{h(x)}\left(\frac{t+s}{t}\right)^3 e^{-\frac{1}{2}(r(\mu_1-1)-\lambda)s}D^{(x)}_s \,,
	$$ from where the result now follows upon taking expectation on both sides of the inequality.
\end{proof} 

\subsection{Conclusion of the proof of the $L^1$ convergence in Theorem~\ref{theo:main3}}\label{sec:conc}

We are now in a condition to conclude the proof of the $L^1$ convergence in Theorem~\ref{theo:main3}. To this end, for each $t \geq 0$ let us choose $s=s(t) \geq 0$ in such a way that the mapping $t \mapsto s(t)$ is continuous, strictly increasing and satisfies $\lim_{t \rightarrow +\infty}\frac{s(t)^{3/2}}{t}=0$. 
Notice that, by Proposition~\ref{theo:main2} with this choice of $s=s(t)$ it will suffice to show that
\begin{equation} \label{eq:norm1}
\lim_{t \rightarrow +\infty} \E_x\left(\left|\frac{|N_{t+s}(B)|}{\E_x(|N_{t+s}|)} - \nu(B) \cdot D^{(x)}_s\right|\right)=0 \,.
\end{equation}

To this end, we set (notice the difference from $\tilde{N}_t^{(x),\{M,s\}}$ in~\eqref{eq:NTrunc}),
$$
\tilde{N}^{(x),[M,s]}_{t}(B):=\{ u \in A^{(x)}_{t} : u_z \in J^{(M)}_z \text{ for all }z \in [0,s]\,,\,u_{t} \in B\} \,.
$$
Then, using the triangle inequality, for each $M > 0$ we can bound the expectation in \eqref{eq:norm1} by the sum of five separate terms $(A)^{M}_{t+s}+(B)^{M}_{t+s}+(C)^{M}_{t+s} +(D)^M_{s,t} + (E)^M_{s,t}$, where:
$$
(A)^{M}_{t+s}:=\E_x\left(\frac{|N_{t+s}(B)|}{\E_x(|N_{t+s}|)}-\frac{|\tilde{N}^{M}_{t+s}(B)}{\E_x(|N_{t+s}|)}\right) \,,
$$ 
$$
(B)^{M}_{t+s}:= \sqrt{\E_x\left( \left|\frac{|\tilde{N}^{M}_{t+s}(B)|}{\E_x(|N_{t+s}|)} - \frac{\E_x(|\tilde{N}^{M}_{t+s}(B)||\mathcal{G}_s)}{\E_x(|N_{t+s}|)}\right|^2\right)} \,,
$$ 
$$
(C)^{M}_{t+s}:= \E_x\left(\frac{\E_x(|\tilde{N}_{t+s}^{[M,s]}(B)||\mathcal{G}_s)}{\E_x(|N_{t+s}|)}-\frac{\E_x(|\tilde{N}^{M}_{t+s}(B)||\mathcal{G}_s)}{\E_x(|N_{t+s}|)}\right) \,,
$$
$$
(D)^M_{s,t}:=\E_x\left(\left| \frac{\E_x(|\tilde{N}_{t+s}(B)^{[M,s]}||\mathcal{G}_s)}{\E_x(|N_{t+s}|)} - \frac{\nu(B)}{1+\varepsilon(x,t+s)}\left(\frac{t+s}{t}\right)^{\frac{3}{2}}\tilde{D}^{(x),M}_s\right|\right) \,,
$$ 
$$
(E)^M_{s,t} := 
\E_x \left(\left|
\frac{\nu(B)}{1+\varepsilon(x,t+s)}\left(\frac{t+s}{t}\right)^{\frac{3}{2}}\tilde{D}^{(x),M}_s
 - \nu(B) \cdot D_s^{(x)} \right| \right) \,.
$$

Now if $M$ is taken sufficiently large (depending only on $x$) then by Proposition \ref{prop:n1} the term $(A)^M_{t+s}$ can be made arbitrarily small for $t$ large enough. Thanks to Proposition~\ref{prop:n3}, the same is true for $(B)^M_{t+s}$. Moreover, by the fundamental property of conditional expectation one has that $(C)^{M}_{t+s} \leq (A)^M_{t+s}$ and thanks to Proposition~\ref{prop:d1}, by taking $M$ large we can guarantee that $(E)^M_{s,t}$ is arbitrarily small for all $t$ large enough. It therefore remains to control the remaining term $(D)_{s,t}$.  

To do this, we note that
by Lemma~\ref{lema:mt1} and Lemma~\ref{lemma:asymp},
$$
\E_x(|\tilde{N}_{t+s}^{[M,s]}(B)||\mathcal{G}_s) = \sum_{u \in \tilde{N}^{(x),M}_s} \E_{u_s}(|N_t(B)|) = \sum_{u \in \tilde{N}^{(x), M}_s} h(u_s)t^{-\frac{3}{2}}e^{(r(\mu_1-1)-\lambda)t}(\nu(B)+\varepsilon_B(u_s,t))
$$ and
$$
\E_x(|N_{t+s}|)=h(x)(t+s)^{-\frac{3}{2}}e^{(r(\mu_1-1)-\lambda)(t+s)}(1+\varepsilon(x,t+s)) \,.
$$ 
Then by a straightforward computation using Lemma~\ref{lema:mt1} again and the bound on $\varepsilon_B(u_s, t)$ from Lemma~\ref{lemma:asymp} we have,
$$
(D)_{s,t}^M \leq  \frac{C}{1+\varepsilon(x,t+s)}\left(\frac{t+s}{t}\right)^{\frac{3}{2}} \frac{(M(1+s^{3/4}) + 1)^2}{t} \E_x (D_s^{M}) \,.
$$ 
In light of the restrictions on $s(t)$ and since $\E_x (D_s^{M}) \leq \E_x (D_s) = 1$, we see that $\lim_{t \rightarrow+\infty} (D)^M_{s,t} = 0$ and the result follows.

\section{Proof of Theorem \ref{theo:main3} - Almost Sure Convergence}\label{sec:pf3}

In this final section we prove the first part of Theorem~\ref{theo:main3}, namely we show that~\eqref{eq:convmain2} holds almost-surely simultaneously for all $B \in \B_{(0,+\infty)}$ with $\nu(\partial B) = 0$. We notice that, in order to so, it will suffice to show that for each $a \in \mathbb{Q}_{\geq 0}$ one has the almost sure convergence
\begin{equation}
\label{eq:convint}
\frac{|N_t^{(x)}((a,+\infty))|}{\E_x(|N_t|)}\overset{a.s.}{\longrightarrow} \nu((a,+\infty)) \cdot D_\infty^{(x)}.
\end{equation} Indeed, from \eqref{eq:convint} it will follow that there exists a full probability event $\tilde{\Omega}_x$ such that 
\begin{equation}
\label{eq:convint2}
\lim_{t \rightarrow +\infty} \frac{|N_t^{(x)}((a,+\infty))|(\omega)}{\E_x(|N_t|)} = \nu((a,+\infty)) \cdot D_\infty^{(x)}(\omega)
\end{equation} holds for all $\omega \in \tilde{\Omega}^{(x)}$ and $a \in \mathbb{Q}_{\geq 0}$. Since $\nu$ is absolutely continuous, by comparison arguments one can extend this to all $a \in \R_{\geq 0}$. In particular, for all $\omega \in \tilde{\Omega}^{(x)} \cap \{ D_\infty^{(x)} > 0\} \cap \{|N_t|> 0 \text{ for all }t \geq 0\}$ and $a \geq 0$, we obtain that
\begin{equation} \label{eq:convint3}
\lim_{t \rightarrow +\infty} \nu^{(x)}_t((a,+\infty))(\omega) = \lim_{t \rightarrow +\infty} \frac{|N_t^{(x)}((a,+\infty))|(\omega)}{\E_x(|N_t|)} \cdot \frac{\E_x(|N_t|)}{|N_t^{(x)}|(\omega)}
=  \nu((a,+\infty)).
\end{equation} By standard properties of weak convergence of probability distributions, \eqref{eq:convint3} can be extended to any $B \in \B_{(0,+\infty)}$ with $\nu(\partial B)=0$. We conclude that for any such $\omega$ one has
$$
\lim_{t \rightarrow +\infty} \frac{|N_t^{(x)}(B)|(\omega)}{\E_x(|N_t|)} = \lim_{t \rightarrow +\infty} \frac{|N_t^{(x)}(B)|(\omega)}{|N_t^{(x)}|(\omega)}\cdot \frac{|N_t^{(x)}|(\omega)}{\E_x(|N_t|)} = \nu(B) \cdot D_\infty^{(x)}(\omega) \,,
$$ for all $B \in \B_{(0,+\infty)}$ with $\nu(\partial B)=0$. Thus, by taking 
$$
\Omega^{(x)}:=\tilde{\Omega}^{(x)} - \{D^{(x)}_\infty > 0\} \triangle \{|N_t| > 0 \text{ for all }t \geq 0\},
$$ where $\triangle$ stands for symmetric difference, \eqref{eq:d2} shows that $\Omega^{(x)}$ is a full probability event and hence the result now follows.

We will prove \eqref{eq:convint} in two steps. First, we shall establish the convergence along sequences $(t_k)_{k \in \N}$ with  vanishing gaps $\Delta_k:=t_{k+1}-t_k$ and then, in a second step, use this convergence to obtain the full limit $t \rightarrow +\infty$. We devote a separate section to each of these steps.

\subsection{Convergence along sequences $(t_k)_{k \in \N}$ with vanishing gaps} 
\label{ss:subseq}
The purpose of this section is to prove the following result.

\begin{proposition} \label{prop:conv1} There exists a sequence $(t_k)_{k \in \N} \subseteq \R_{\geq 0}$ satisfying:
	\begin{enumerate}
		\item [T1.] $\lim_{k \rightarrow +\infty} t_k = +\infty,$
		\item [T2.] $\lim_{k \rightarrow +\infty} (t_{k+1}-t_k) = 0,$
		\item [T3.] $\sum_{k \in \N} e^{-\frac{1}{4}(r(\mu_1-1)-\lambda)t_k} < +\infty$,
	\end{enumerate} such that  for all $B \in \B_{(0,+\infty)}$ and $x > 0$ one has
\begin{equation} \label{eq:convprop1}
\frac{|N^{(x)}_{t_k}(B)|}{\E_x(|N_{t_k}|)} \overset{a.s.}{\longrightarrow} \nu(B) \cdot D^{(x)}_\infty \,,
\end{equation} as $k \rightarrow +\infty$.  
\end{proposition}
\begin{proof}
We choose $\tilde{t}_k := (\log k)^{10}$, $s_k := (\log k)^{4}$ and $M_k := \delta \log k$, where $\delta > 0$ is the constant from Proposition \ref{prop:n2}, and then set $t_k:=\tilde{t}_k+s_k$. It is straightforward to check that, if chosen in this way, the sequence $(t_k)_{k \in \N}$ satisfies (T1)-(T2)-(T3) in the statement of the proposition. Thus, it remains to show \eqref{eq:convprop1}. 

To this end, as in Subsection~\ref{sec:conc},
we decompose 
$$
\frac{|N^{(x)}_{t_k}(B)|}{\E_x(|N_{t_k}|)} - \nu(B) \cdot D_\infty^{(x)}= \sum_{i=1}^5 [i]_{k}
\,,$$ where:
$$
[1]_{k} := \frac{|N^{(x)}_{t_k}(B)|}{\E_x(|N_{t_k}|)} - \frac{|\tilde{N}^{(x),M_k}_{t_k}(B)|}{\E_x(|N_{t_k}|)} \,,
$$
$$
[2]_{k}:=  \frac{|\tilde{N}^{(x),M_k}_{t_k}(B)|}{\E_x(|N_{t_k}|)} -  \frac{\E_x(|\tilde{N}^{M_k}_{t_k}(B)||\mathcal{G}_{s_k})}{\E_x(|N_{t_k}|)} \,,
$$
$$
[3]_{k}:= \frac{\E_x(|\tilde{N}^{M_k}_{t_k}(B)||\mathcal{G}_{s_k})}{\E_x(|N_{t_k}|)} - \frac{\E_x(|\tilde{N}_{t_k}^{[M_k,s_k]}(B)||\mathcal{G}_{s_k})}{\E_x(|N_{t_k}|)} \,,
$$
$$
[4]_k:=  \frac{\E_x(|\tilde{N}_{t_k}^{[M_k,s_k]}(B)||\mathcal{G}_{s_k})}{\E_x(|N_{t_k}|)} - \frac{\nu(B)}{1+\varepsilon(x,t_k)}\left(\frac{t_k}{\tilde{t}_k}\right)^{\frac{3}{2}} \tilde{D}^{(x),M_k}_{s_k} \,,
$$
$$
[5]_k:=  \frac{\nu(B)}{1+\varepsilon(x,t_k)}\left(\frac{t_k}{\tilde{t}_k}\right)^{\frac{3}{2}} \tilde{D}^{(x),M_k}_{s_k} - \nu(B) \cdot D^{(x)}_\infty \,.
$$
Notice that it will suffice to show that each term $[i]_k$ converges almost surely to zero as $k \rightarrow +\infty$.

Now, by Chebyshev's inequality together with the $L^1$-bound obtained in Proposition \ref{prop:n1}, it follows from the Borel-Cantelli Lemma and the choice of $M_k$ that $[1]_k \rightarrow 0$ almost surely. Similarly, by the $L^2$-bound obtained in Proposition \ref{prop:n3}, the choice of $s_k$ and the fact that $s_k \ll \tilde{t}_k$, the Borel-Cantelli Lemma yields that $[2]_k \rightarrow 0$ almost surely as well. On the other hand, upon noticing that
$$
|[3]_k| \leq
\frac{\E_x(|N_{t_k}(B)||\mathcal{G}_{s_k})}{\E_x(|N_{t_k}|)} -  \frac{\E_x(|\tilde{N}^{M_k}_{t_k}(B)||\mathcal{G}_{s_k})}{\E_x(|N_{t_k}|)}, 
$$ the same $L^1$-bound used for $[1]_k$ can also be applied here (recall the equality $(C)^{M_k}_{t_k} = (A)^{M_k}_{t_k} $ from Section \ref{sec:conc}) to conclude that $[3]_k \rightarrow 0$ almost surely. Also, since $t_k \sim \tilde{t}_k$ as $k \rightarrow +\infty$, that $[5]_k \rightarrow 0$ almost surely follows in the same fashion as $[1]_k$, using instead the $L^1$-bound from Proposition \ref{prop:d1}. Hence, it remains to check that $[4]_k \rightarrow 0$. 

To do this we observe that, by branching property and Lemmas \ref{lema:mt1} and~\ref{lemma:asymp}, we have
\begin{align*}
\E_x(|\tilde{N}_{t_k}^{[M_k,s_k]}(B)||\mathcal{G}_{s_k}) & = \sum_{u \in \tilde{N}^{(x),M_k}_{s_k}} \E_{u_s}(N_{\tilde{t}_k}(B)) \\
& = (\tilde{t}_k)^{-\frac{3}{2}}e^{(r(\mu_1-1)-\lambda)\tilde{t}_k}\sum_{u \in \tilde{N}^{(x),M_k}_{s_k}} h(u_{s_k})(\nu(B)+\varepsilon_B(u_{s_k},\tilde{t}_k)) \,,
\end{align*} so that, by these lemmas again and a straightforward computation, we get
$$
[4]_k = \frac{1}{1+\varepsilon(x,t_k)}\left(\frac{t_k}{\tilde{t}_k}\right)^{\frac{3}{2}} \frac{1}{h(x)}\sum_{u \in \tilde{N}^{(x),M_k}_{s_k}} h(u_{s_k})e^{-(r(\mu_1-1)-\lambda)s_k}\varepsilon_B(u_{s_k},\tilde{t}_k).
$$ But since $u_{s_k} \leq M_k(1+s_k^{\frac{3}{4}}) \leq 2\delta s_k$ for any $u \in \tilde{N}^{(x),M_k}_{s_k}$ if $k$ is sufficiently large, by \eqref{eq:cotaeb} we obtain that for any such $k$
\begin{equation} \label{eq:bound5}
|[4]_k| \leq C \frac{1}{1+\varepsilon(x,t_k)} \left(\frac{t_k}{\tilde{t}_k}\right)^{\frac{3}{2}} \cdot D^{(x)}_{s_k} \cdot \frac{s_k^2}{\tilde{t}_k} \,,
\end{equation} for some constant $C > 0$ depending only on $r,\mu$ and $\lambda$. Therefore, since $t_k \sim \tilde{t}_k$, $s_k^2 \ll \tilde{t}_k$ and $D^{(x)}_{s_k} \rightarrow D^{(x)}_\infty < +\infty$ almost surely as $k \rightarrow +\infty$, from the bound \eqref{eq:bound5} we conclude that $[4]_k \rightarrow 0$ almost surely as $k \rightarrow +\infty$ and thus Proposition \ref{prop:conv1} now follows.
\end{proof}

\subsection{The full limit}
\label{ss:full}

Recall that by Lemma~\ref{lemma:asymp} we have $\E_x (|N_t|) \sim n_x(t)$, where 
$n_x(t) := h(x) t^{-3/2} e^{r((\mu_1 - 1) - \lambda)t}$. Thanks to Proposition \ref{prop:conv1} we can then assume the existence of a full probability event on which we have, 
\begin{equation}
\label{eq:exlim}
\lim_{k \rightarrow +\infty} \frac{|N^{(x)}_{t_k}(I_a)|}{n_x(t_k)}= \nu(I_a)\cdot D_\infty^{(x)} \,,
\end{equation} for \textit{all} $a \in \mathbb{Q}_{\geq 0}$, where we denote $I_a:=(a,+\infty)$. 
Therefore in order to prove~\eqref{eq:convint} it suffices to show that for each $a \in \mathbb{Q}_{\geq 0}$ and every $\epsilon > 0$, with probability $1$ there exists some (random) $k_0 \in \N$ such that for all $k \geq k_0$,
\begin{equation}
\label{eq:interp}
\sup_{s \in [t_k, t_{k+1})} \frac{|N_s^{(x)}(I_a)|}{n_x(s)}
- 
\frac{|N_{t_k}^{(x)}(I_a)|}{n_x(t_k)} \leq \epsilon
\hspace{0.7cm}\text{ and }\hspace{0.7cm}
\frac{|N_{t_k}^{(x)}(I_a)|}{n_x(t_k)} - 
\inf_{s \in [t_k, t_{k+1})} \frac{|N_s^{(x)}(I_a)|}{n_x(s)}
\leq \epsilon,
\end{equation} 

To deal with the first inequality in \eqref{eq:interp}, we fix any $a \in \mathbb{Q}_{\geq 0}$ and let $a' := \alpha a$, with $\alpha \in (0,1) \cap \mathbb{Q}$ to be specified later. We also define
$\cN^{(y)}_{t}(B) := \max_{s \in [0,t]} |N^{(y)}_s(B)|$ for any $y,t \geq 0$ and $B \in \B_{(0,+\infty)}$ and use the earlier convention that $\cN^{(y)}_t := \cN^{(y)}_t((0,\infty))$. Finally recall that $\Delta_k := t_{k+1} - t_k$ for each $k \in \N$. Using this notation, the branching property yields the bound
\begin{equation}\label{eq:bound1}
\sup_{s \in [t_k, t_{k+1})} |N^{(x)}_s(I_a)| - |N^{(x)}_{t_k}(I_a)|
\, \leq \sum_{u \in N^{(x)}_{t_k}(I_{a'})} \cN_{\Delta_k}^{(u_{t_k})}(I_a)
+ \sum_{u \in N^{(x)}_{t_k}(I_{a'}^c)} \cN_{\Delta_k}^{(u_{t_k})}(I_a) 
- \sum_{u \in N^{(x)}_{t_k}(I_a)} 1\,,
\end{equation}
where, as in the past, all terms appearing in the sums on the right-hand side of \eqref{eq:bound1} are independent conditionally on $\mathcal{G}^{(x)}_{t_k}$. The right-hand side above can be further written as
\begin{multline}
\label{e:32}
\big|N_{t_k}^{(x)}(I_{a'} \setminus I_a) \big| + 
\sum_{u \in N^{(x)}_{t_k}(I_{a'})} \Big(\cN_{\Delta_k}^{(u_{t_k})} - \E_{u_{t_k}}( \cN_{\Delta_k})\Big)
+ \sum_{u \in N^{(x)}_{t_k}(I_{a'})} \Big( \E_{u_{t_k}} (\cN_{\Delta_k}) - 1 \Big) \\
+ \sum_{u \in N^{(x)}_{t_k}(I_{a'}^c)} \Big(\cN_{\Delta_k}^{(u_{t_k})}(I_a) - 
\E_{u_{t_k}} (\cN_{\Delta_k}(I_a)) \Big)
+ \sum_{u \in N^{(x)}_{t_k}(I_{a'}^c)} \E_{u_{t_k}} (\cN_{\Delta_k}(I_a)) \,.
\end{multline}

Now, by comparing $\cN^{(y)}_t$ with the number of individuals in $\tilde{A}^{(y)}_t$, defined as $A^{(y)}_y$ only for a modified branching Brownian motion process, in which the offspring distribution is $\tilde{\mu}(\cdot) := \mu(\cdot -1)$ (i.e. with $m$ replaced by $m+1$ everywhere) we have $\sup_{y > 0} \E_y (\cN_t) \leq e^{r \mu_1t} \downarrow 1$ as $t \downarrow 0$. Similarly, for all $0 \leq y \leq a'$ we have
$$
\E_y (\cN_t(I_a)) \leq \E_{a'} (\cN_t(I_a)) \leq e^{r \mu_1 t}P_{a'}\left( \sup_{s \in [0,t]} X_s > a \right) \downarrow 0
\quad \text{as } t \downarrow 0 \,.
$$ 
Using the fact that $\Delta_k \to 0$ as $k \to +\infty$ as implied by our assumptions on the sequence $(t_k)_{k \in \N}$, the first, third and last terms in~\eqref{e:32} can therefore be bounded together by
\begin{equation}\label{eq:bound2}
\big|N_{t_k}^{(x)}(I_{a'} \setminus I_a) \big| + \delta \big|N_{t_k}^{(x)}(I_{a'}) \big|
+ \delta \big|N_{t_k}^{(x)}(I_{a'}^c) \big| = \big|N_{t_k}^{(x)}(I_{a'} \setminus I_a) \big| + \delta |N^{(x)}_{t_k}| \,,
\end{equation}
for $\delta > 0$ arbitrarily small, provided $k$ is large enough. 

Thanks then to the existence of \mbox{the limit in \eqref{eq:exlim}} for $I_b$ with $b=a,a',0$, the right hand side of \eqref{eq:bound2} can be bounded by $((\nu(I_{a'} \setminus I_\alpha) + \delta) D^{(x)}_\infty + 2\delta) \E_x(|N_{t_k}|)$ for all $k$ large enough almost surely. Finally, using the almost sure finiteness of $D^{(x)}_\infty$, we can choose $\alpha$ close enough to $1$ and $\delta$ close enough to $0$, so that the sum of the first, third and last terms in~\eqref{e:32} can be made at most $(\epsilon/4) \E_x(|N_{t_k}|)$ for all large enough $k$ with probability $1$.

Turning to the remaining terms in \eqref{e:32}, by comparison with $\tilde{A}^{(y)}_t$ again and using the many-to-two lemma, we have $\sup_{y > 0, t \in [0,1]} \E_y(\cN_t^2) \leq e^{2r\mu_1+r[\text{Var}(m)+\mu_1^2]} < +\infty$. Therefore, by conditioning on $\cG_{t_k}$ and then using the independence of the random variables $\{\cN_{\Delta_k}^{(u_{t_k})}: u \in N^{(x)}_{t_k}\}$ given $\mathcal{G}_{t_k}$, the $L^2$-norm of the second term in~\eqref{e:32} can bounded from above by the square root of 
\begin{equation}
\nonumber
\E_x \left( \sum_{u \in N_{t_k}(I_{a'})} \E_{u_{t_k}} \Big(\big(\cN_{\Delta_k} - \E_{u_{t_k}} (\cN_{\Delta_k})\big)^2 \,\big|\, \cG_{t_k} \Big) \right)
\leq (\delta')^{-1} \E_x(|N_{t_k}(I_{a'})|) \leq (\delta')^{-1} \E_x(|N_{t_k}|) \,,
\end{equation}
for some $\delta' > 0$ and any $k$ large enough. Moreover, since the second moment of each $\cN^{(u_{t_k})}_{\Delta_k}(I_a)$ is even smaller, the same bound also holds for the fourth term in~\eqref{e:32}. 

Altogether by Markov's inequality, the probability that the sum of the second and fourth terms in~\eqref{e:32} exceeds $(\epsilon/4) \E_x(|N_{t_k}|)$ is at most
$8 (\delta')^{-1/2} \epsilon^{-1} \big(\E_x(|N_{t_k}|)\big)^{-1/2}$. In light of \eqref{eq:sc2} and condition (T3) on the growth of the sequence $(t_k)_{k \in \N}$, the latter probability is summable in $k$ and hence, by the Borel-Cantelli Lemma, the latter event ceases to occur after some random (but finite) $k$. Together with the previous bounds this shows that, with probability $1$, eventually one has
\begin{equation}
\sup_{s \in [t_k, t_{k+1})} |N^{(x)}_s(I_a)| - |N^{(x)}_{t_k}(I_a)| \leq (\epsilon/2) \E_x (|N_{t_k}|) \,.
\end{equation} 
Dividing by $n_x(t_k)$, using the monotonicity of $s \mapsto n_x(s)$ for $s$ large enough and recalling that $\E_x (|N_{t_k}|) \sim n_x(t_k)$, then yields the left inequality in~\eqref{eq:interp} for all $k$ large enough.

The argument for the right inequality in~\eqref{eq:interp} goes along the same lines. This time we let $a' \in (a,+\infty) \cap \mathbb{Q}$ to be determined later, and bound $|N^{(x)}_{t_k}(I_a)| - \inf_{s \in [t_k, t_{k+1})} |N^{(x)}_s(I_a)|$ from above by
\begin{equation}
\label{eq:36}
\big|N^{(x)}_{t_k}(I_a \setminus I_{a'})\big| +
\sum_{u \in N^{(x)}_{t_k}(I_{a'})} \Big(\mathbbm{1}_{\{\cN^{(u_{t_k})}_{\Delta_k}(\{a\}) \neq 0\}} - 
P_{u_{t_k}}\big(\cN_{\Delta_k}(\{a\}) \neq 0\big) \Big) +
\sum_{u_{t_k} \in N^{(x)}_{t_k}(I_{a'})} P_{u_{t_k}} \big(\cN_{\Delta_k}(\{a\}) \neq 0\big),
\end{equation} where $\cN_{\Delta_k}^{(y)}(\{0\}) \neq 0$ here simply means that at least one particle in $A^{(y)}_{\Delta_k}$ has been absorbed at $0$.
As before, for all $y \geq a' > a$ we have that as $t \downarrow 0$
$$
P_y \big(\cN_{t}(\{a\}) \neq 0\big) \leq P_{a'}\big(\cN_{t}(\{a\}) \neq 0\big) \leq e^{r \mu_1 t}P_{a'}\left( \inf_{s \in [0,t]} X_s < a\right) \downarrow 0 \,.
$$ Also, we have $\nu(I_a \setminus I_{a'}) \downarrow 0$ as $a' \downarrow a$. Therefore, by choosing first $a'$ close enough to $a$, then taking $k$ sufficiently large and finally using~\eqref{eq:exlim} for $I_b$ with $b=a,a'$ and the finiteness of $D^{(x)}_\infty$, the first and last terms of \eqref{eq:36} will eventually be bounded together by $(\epsilon/4) \E_x (|N_{t_k}|)$ almost surely.

At the same time, by conditioning on $\cG_{t_k}$ as before and using the conditional independence of the terms, the second moment of the middle term in \eqref{eq:36} is at most $4 \E_x (|N_{t_k}(I_{a'})|) \leq 4 \E_x(|N_{t_k}|)$. Then, by Markov's inequality, the probability that this term exceeds $(\epsilon/4) \E_x (|N_{t_k}|)$ is at most $8 \epsilon^{-1} \left(\E_x (|N_{t_k}|)\right)^{-1/2}$ which is again summable in $k$ by~\eqref{eq:sc2} and (T3). Invoking the Borel-Cantelli Lemma again, this shows that, almost surely, also the middle term eventually becomes less than $(\epsilon/4) \E_x (|N_{t_k}|)$ and hence 
$|N^{(x)}_{t_k}(I_a)| - \inf_{s \in [t_k, t_{k+1})} |N^{(x)}_s(I_a)|$ will be eventually at most $(\epsilon/2) \E_x(|N_{t_k}|)$.

Finally, thanks to monotonicity of $s \mapsto n_x(s)$ for $s$ large again, the left-hand side of the right inequality in~\eqref{eq:interp} is at most
\begin{equation}\label{eq:50}
\frac{|N_{t_k}^{(x)}(I_a)| - 
	\inf_{s \in [t_k, t_{k+1})} |N_s^{(x)}(I_a)|}
{n_x(t_{k+1})}
+ \frac{|N_{t_k}^{(x)}(I_a)|}{n_x(t_k)}
\Big(1 - \frac{n_x(t_k)}{n_x(t_{k+1})}\Big) \,.
\end{equation}
Since $\E_x(|N_{t_k}|) \sim n_x(t_k) \leq n_x(t_{k+1})$ for $k$ large enough, what we have just shown implies that, almost surely, the first term will eventually be below $3\epsilon/4$. At the same time, while the first factor in the second term tends to a finite limit in light of~\eqref{eq:exlim}, the second factor
goes to $0$, and therefore the second term will become smaller than $\epsilon/4$ for all large enough $k$ almost surely. Together with the bound on the first term, this shows the second inequality in~\eqref{eq:interp} for all $k$ large enough and completes the proof.

\section*{Acknowledgments}
We would like to thank P. Maillard and the anonymous referee for pointing out to us that Proposition~\ref{theo:main2} was already proved in~\cite{Kyp2004}. Santiago Saglietti would also like to thank Bastien Mallein for very helpful discussions on these topics. The work of O.L. was supported in part by the European Union's - Seventh Framework Program (FP7/2007-2013) under grant agreement no. 276923 -- M-MOTIPROX. 
The work of S.S. was supported by the Israeli Science Foundation grant no. 1723/14 -- Low Temperature Interfaces: Interactions, Fluctuations and Scaling.
\bibliographystyle{plain}
\bibliography{biblio4}

\end{document}